\documentclass{article}

\usepackage{amsmath,amssymb,amsfonts}
\usepackage{theorem}
\usepackage{dsfont}
\usepackage{hyperref}

\providecommand{\keywords}[1]
{
  \small	
  \textbf{\textit{Keywords---}} #1
}

\newtheorem{theorem}{Theorem}[section]
\newtheorem{lemma}[theorem]{Lemma}
\newtheorem{proposition}[theorem]{Proposition}
\newtheorem{remark}[theorem]{Remark}
\newtheorem{corollary}[theorem]{Corollary}

\newenvironment{proof}
  {{\bf Proof:}}
  {\qquad \hspace*{\fill} $\Box$}

\newcommand{\N}{\mathbb{N}}
\newcommand{\R}{\mathbb{R}}

\newcommand{\rmd}{\mathrm{d}}
\newcommand{\EC}{\mathcal{E}}
\newcommand{\fa}{\mathfrak{a}}
\newcommand{\FC}{\mathcal{F}}
\newcommand{\SC}{\mathcal{S}}
\newcommand{\MC}{\mathcal{M}}
\newcommand{\VC}{\mathcal{V}}
\newcommand{\tm}{\times}
\newcommand{\trn}{^{\top}}
\newcommand{\tr}{\mathrm{tr}}
\newcommand{\GL}{\mathrm{GL}}
\newcommand{\bary}{\mathrm{bar}}
\newcommand{\mydet}{\mathrm{det}}
\newcommand{\unit}{\mathds{1}}
\newcommand{\ep}{\varepsilon}
\newcommand{\vol}{\mathrm{vol}}
\newcommand{\Diag}{\mathrm{Diag}}

\DeclareMathOperator*{\argmin}{arg\,min}

\allowdisplaybreaks

\begin{document}

\title{A convex optimization approach to the Lyapunov exponents}
\author{Christoph Kawan\footnote{Email: christoph.kawan@gmx.de; no affiliation}\,\,\footnote{Acknowledgements: I express my sincere gratitude to Dr.~Nicola Cavallucci for his continuing interest in my paper and helping me to gain a better understanding of the space of Riemannian metrics and its geometry. Furthermore, special thanks go to the anonymous reviewers for their valuable comments that led to substantial improvements of the paper.}}
\date{}
\maketitle

\begin{abstract}
The aim of this paper is to shed more light on some recent ideas about Lyapunov exponents and clarify the formal structures behind these ideas. In particular, we show that the vector of averaged Lyapunov exponents of a smooth measure-preserving dynamical system can be regarded as the solution to a vector-valued optimization problem on a space $\MC$ of Riemannian metrics. Similar results were first proved by Jairo Bochi and Andr\'es Navas in the language of linear cocycles and their conjugacies. We go one step further and prove that the optimization problem is geodesically convex with respect to the $L^2$-metric on $\MC$. Moreover, we derive some consequences of this fact.
\end{abstract}


\section{Introduction}

Originally introduced as stability indicators for ordinary differential equations, Lyapunov exponents have become a cornerstone of the theory of smooth dynamical systems. The main reason is that they are related in numerous ways to other important characteristics, in particular to several notions of entropy and dimensions of invariant sets as well as invariant measures. The study of their properties and methods for their computation is an ongoing endeavor. Already in dimension two, examples of dynamical systems are known whose Lyapunov exponents up to this date cannot be effectively estimated on large subsets of the state space. For instance, for the Chirikov standard map, an area-preserving diffeomorphism of the $2$-torus, it is a long-standing open problem to determine whether the maximal Lyapunov exponent is positive on a set of positive Lebesgue measure, see e.g.~\cite{BXY}.%

Lyapunov's first method for proving the asymptotic stability of a system aims at showing that the maximal Lyapunov exponent is negative. His second method uses a non-negative real-valued function on the state space with values decreasing along trajectories. Once a candidate for such a \emph{Lyapunov function} has been found, the decrease condition can be checked by a simple computation involving only the candidate function and the right-hand side of the system's equation. One benefit of this method thus consists in the reduction of the proof of a long-term property to a ``short-term computation''.%

Now, asymptotic stability is a comparatively simple property. Other properties, studied mainly in chaos theory, such as various forms of hyperbolicity, can most likely not be checked by evaluating a single scalar-valued function along trajectories. Since these properties are by nature multi-dimensional as their definitions involve a distinction of different state space directions, a structure which is itself multi-dimensional can help to verify them via simple short-term computations. The main example of such a structure is a Riemannian metric. Usually, a Riemannian metric which allows to verify a long-term property by a short-term computation is called an \emph{adapted metric}. For instance, a classical result states that any uniformly hyperbolic set of a discrete-time system admits an adapted metric (also called \emph{Lyapunov metric}) in which contraction on the stable bundle and expansion on the unstable bundle can be seen in one time step, see \cite[Ch.~6]{HKa}. An analogous result for the more general property of supporting a dominated splitting is proved in \cite{Gou}. In the context of nonuniformly hyperbolic systems, measurable adapted metrics can be constructed, see \cite[p.~668]{KHa}. Finally, in the context of stability analysis, besides Lyapunov functions, also Riemannian metrics are used to prove contractivity properties (implying asymptotic stability), which are known as \emph{contraction metrics}, see, e.g., the recent survey \cite{GHK}.%

In \cite{Boc,BNa}, Bochi and Navas introduced a novel method to construct Riemannian metrics adapted to continuous linear cocycles over topological dynamical systems. In this general context, they are not Riemannian metrics in the actual sense (as there is no smooth manifold involved), but rather continuous mappings assigning a positive-definite matrix to each point in the state space. They can then be interpreted as conjugacies between cocycles. In the special case of the derivative cocycle of a smooth map, they become Riemannian metrics in the actual sense, which allow to approximate the averaged Lyapunov vector (with respect to an invariant probability) by integration over the vector of (log-) singular values of the right-hand side Jacobian. In this paper, we provide a self-contained proof of this result in the context of the derivative cocycle. An important interpretation of the result is that the averaged Lyapunov vector is the infimum of a vector-valued function $\FC$ defined on the space of all continuous Riemannian metrics. Here, the infimum needs to be understood with respect to the so-called majorization order. It is important to note that it is not only a weak, i.e.~a Pareto infimum, but an infimum in all components. Hence, the averaged Lyapunov vector appears as the solution to a vector-valued optimization problem.%

We further prove that the optimization problem posed for $\FC$ is convex in a proper sense. First, one can restrict $\FC$ to the space $\MC$ of \emph{smooth} Riemannian metrics which is a Fr\'echet manifold modeled on the Fr\'echet space of smooth symmetric $(0,2)$-tensor fields. The most-studied Riemannian metric on this manifold is the $L^2$-metric, which is defined by integration over the fiberwise trace metric (the metric of the symmetric space of positive-definite matrices). The $L^2$-metric on $\MC$ has been studied for a long time in different contexts; the main references include \cite{Cav,Cla,Cl2,Ebi,FGr,GMi}. Though $\MC$ is only a weak Riemannian manifold, the central objects studied in Riemannian geometry (including geodesics) can be shown to exist and explicit formulas for them are available. An important submanifold of $\MC$ is the space $\MC_{\omega}$ of metrics inducing the same (arbitrary) volume form $\omega$. It was proved in \cite{FGr} that $\MC_{\omega}$ is a symmetric space with the inherited $L^2$-metric; in particular, it is geodesically complete. The function $\FC$ has the property that its infimum is preserved under restriction to $\MC_{\omega}$. Hence, the search for a (nearly) optimal metric can be restricted to $\MC_{\omega}$. We prove that the restriction of $\FC$ to $\MC_{\omega}$ is geodesically convex with respect to the $L^2$-geodesics on its domain and the majorization order on its co-domain. In fact, this proof was already given in \cite{KHG}, but without putting it in the proper context and relating it to the geometry of $\MC$.%

In summary, we prove that the averaged Lyapunov vector is the infimum of a geodesically convex vector-valued function. The associated optimization problem has some serious disadvantages, however. It is not guaranteed that a minimizer exists, the problem is infinite-dimensional, defined on a Riemannian manifold (rather than a vector space) and nonsmooth. We elaborate on some of these aspects and make suggestions how to mitigate the corresponding problems.%

Several other works have proposed convex optimization to compute Lyapunov exponents or associated quantities. In particular, we point the reader to the paper \cite{OGo}, where SOS programming is used to find tight bounds on the maximal Lyapunov exponent of a system given by an ordinary differential equation whose right-hand side is polynomial. This approach was generalized to the computation of sums of Lyapunov exponents and dimensions of attractors in \cite{PGo}. The optimization there is not over Riemannian metrics, but over auxiliary real-valued functions. The paper \cite{Ani} also studies adapted metrics for the estimation of (uniform) Lyapunov exponents, but on exterior bundles of the state space instead of the tangent bundle. This method is also used for obtaining effective dimension estimates for several classes of systems, and refined in \cite{ARo} to obtain estimates for topological entropy. We finally mention the papers \cite{Gea,KHG,Lea}, in which the groundwork has been laid for the approach pursued in this paper, and numerical algorithms were developed to use it for the computation of contraction metrics \cite{Gea}, the estimation of Lyapunov dimension \cite{Lea} and the estimation of restoration entropy \cite{KHG}.%

Organization of the paper: In Section \ref{sec_prelim}, we provide the necessary background on Lyapunov exponents. Section \ref{sec_bochi} is devoted to the fundamental result about the approximation of the Lyapunov vector by singular values of the time-one Jacobian, see Theorem \ref{thm_bochi}. In Section \ref{sec_riem_metrics}, we introduce the $L^2$-metric on the space of Riemannian metrics and present some results on the structure of this space. The subsequent Section \ref{sec_convexity} contains the convexity proof of the function $\FC$, Theorem \ref{thm_convexity}. In Theorem \ref{thm_lipschitz_l2}, we also prove the global Lipschitz continuity of the components of $\FC$. In Section \ref{sec_first_order}, we provide a formula for the directional derivatives of the components of $\FC$ in an idealized case, where the existence of these derivatives is guaranteed, see Theorem \ref{thm_derivative}. The final Section \ref{sec_final} provides a summary and an outlook.

\section{Preliminaries}\label{sec_prelim}

\subsection{Notation}

By $\log(\cdot)$, we denote the natural logarithm. If $A$ is a square matrix, we write $A\trn$ and $\tr(A)$ for the transpose and the trace of $A$, respectively. If $x \in \R^d$, then $\|x\|_p$ denotes its $p$-norm for $p \in \{1,2,\infty\}$. The components of $x$ will always be denoted by $x_1,\ldots,x_d$. By $\unit$, we denote the vector in $\R^d$ whose components are all equal to $1$. We write $\GL(d,\R)$ for the general linear group of $\R^d$. If $f:M \hookleftarrow$ is a differentiable self-map of a smooth manifold $M$, we write $fx$ and $\rmd f_x$ for its value and its derivative at $x \in M$, respectively. By $I$, we denote the identity matrix of any dimension. By a \emph{metric} $g$ on a smooth manifold $M$, we always mean a Riemannian metric, i.e.~the assignment of an inner product $g_x$ to each tangent space $T_xM$, which is at least continuous. We use the notation $\MC^0$ ($\MC$) for the space of all continuous (smooth) Riemannian metrics on a given manifold $M$.%

\subsection{Setup}

We consider a discrete dynamical system given by a $C^1$-diffeomorphism $f:M \hookleftarrow$ on a compact orientable\footnote{The orientability of $M$ is a technical assumption which guarantees that we can work with volume forms.} smooth manifold $M$ of finite dimension $d$. We further assume that $f$ preserves a Borel probability measure $\mu$, i.e.~$\mu(f^{-1}(A)) = \mu(A)$ for all Borel sets $A \subset M$.%

\subsection{Background on Lyapunov exponents}

Our objects of interest are the Lyapunov exponents associated with the data $(M,f,\mu)$. The existence of these numbers is guaranteed by the famous \emph{Multiplicative Ergodic Theorem (MET)}, also known as \emph{Oseledets Theorem}. This theorem includes quite a number of statements and can be formulated on different levels of generality. Here, we formulate a version that is adapted to our setup and our needs. First, we define the \emph{Lyapunov exponent} at a point $x \in M$ in direction $0 \neq v \in T_xM$ as the limit%
\begin{equation*}
  \lambda(x,v) := \lim_{n \rightarrow \infty}\frac{1}{n}\log\|\rmd f^n_xv\|,%
\end{equation*}
provided that this limit exists. Here, $\|\cdot\|$ denotes the norm (on each tangent space) induced by a given Riemannian metric. The value of $\lambda(x,v)$, however, is independent of the chosen metric as a consequence of compactness.%

\begin{theorem}{\bf (MET)}
In the given setup, for $\mu$-almost every $x \in M$ there exist numbers%
\begin{equation*}
  \lambda_1(x) \geq \lambda_2(x) \geq \cdots \geq \lambda_d(x)%
\end{equation*}
such that the following statements hold:%
\begin{enumerate}
\item[(i)] For every tangent vector $0 \neq v \in T_xM$, $\lambda(x,v) = \lambda_i(x)$ for some $i$.%
\item[(ii)] $\lambda_1(x) + \lambda_2(x) + \cdots + \lambda_d(x) = \lim_{n\rightarrow\infty}(1/n)\log|\det\rmd f^n_x|$.%
\item[(iii)] The functions $\lambda_i(\cdot)$ are Borel measurable and $f$-invariant.%
\item[(iv)] Let $\alpha_1(\rmd f^n_x) \geq \alpha_2(\rmd f^n_x) \geq \cdots \geq \alpha_d(\rmd f^n_x)$ denote the singular values of the linear operator $\rmd f^n_x:T_xM \rightarrow T_{f^nx}M$. Then the limit%
\begin{equation*}
  \lim_{n \rightarrow \infty}\frac{1}{n}\log\alpha_i(\rmd f^n_x)%
\end{equation*}
exists for each $i$ and equals $\lambda_i(x)$.%
\item[(v)] The convergence in (iv) also holds in the $L^1$-norm, i.e.%
\begin{equation*}
   \lim_{n \rightarrow \infty}\int \left|\frac{1}{n}\log\alpha_i(\rmd f^n_x) - \lambda_i(x)\right|\, \rmd\mu(x) = 0,\quad i = 1,2,\ldots,d.%
\end{equation*}
\end{enumerate}
If $\mu$ is ergodic, then the functions $\lambda_i(\cdot)$ are constant almost everywhere.
\end{theorem}

We base the analysis in this paper on statements (iv) and (v) above, i.e.~on the relation between the Lyapunov exponents and the growth rates of singular values. In contrast to eigenvalues, singular values are geometric quantities; they depend on the inner products on domain and co-domain of the linear operator in question. In our setup, this means that they depend on the Riemannian metric imposed on $M$, while their exponential growth rates, the Lyapunov exponents, do not. Hence, one can ask whether there exists a metric adapted to the given system such that in this metric the Lyapunov exponents can be computed (at least approximately) as the expansions rates in the first iterate, i.e.%
\begin{equation*}
  \lambda_i(x) \approx \log\alpha_i(\rmd f_x).%
\end{equation*}
The existence of such metrics essentially follows from a result proved in \cite[Prop.~4.1]{Boc} about linear cocycles over topological dynamical systems in combination with methods from \cite{BNa}. In the language of linear cocycles, switching to another metric corresponds to switching to a conjugate cocycle. For the convenience of the reader, we provide a complete proof of the result within the setup introduced above.

\section{Existence of adapted metrics}\label{sec_bochi}

We first need to introduce some technical concepts and the associated notation. Given an invertible linear operator $L:V \rightarrow W$ between finite-dimensional inner product spaces of dimension $d$, we write%
\begin{equation*}
  \alpha_1(L) \geq \alpha_2(L) \geq \cdots \geq \alpha_d(L)%
\end{equation*}
for the singular values of $L$. Geometrically, these are the lengths of the semi-axes of the ellipsoid $\EC := \{ Lv : \langle v,v \rangle_V = 1 \}$ measured in the norm induced by the inner product on $W$. We then write%
\begin{equation*}
  \vec{\sigma}(L) := (\log\alpha_1(L),\log\alpha_2(L),\ldots,\log\alpha_d(L)).%
\end{equation*}
By our assumption that $L$ is invertible, the numbers $\alpha_i(L)$ are strictly positive, and hence $\vec{\sigma}(L) \in \R^d$. We can be more precise. Because of the ordering of the singular values, $\vec{\sigma}$ takes values in the closed convex cone\footnote{The notation $\fa^+$ is adopted from the theory of semisimple Lie groups, where it is used for the positive Weyl chamber.}%
\begin{equation*}
  \fa^+ := \left\{ \xi \in \R^d : \xi_1 \geq \xi_2 \geq \cdots \geq \xi_d \right\}.%
\end{equation*}
A partial order on $\fa^+$, which is very useful for our purposes, is defined by%
\begin{equation*}
  \xi \preceq \eta \quad :\Leftrightarrow \quad \left\{\begin{array}{rl} \xi_1 + \cdots + \xi_k \leq \eta_1 + \cdots + \eta_k & \mbox{for } k = 1,\ldots,d-1, \\
	\xi_1 + \cdots + \xi_k = \eta_1 + \cdots + \eta_k & \mbox{for } k = d. \end{array}\right.%
\end{equation*}
This order is called the \emph{majorization order}. It can be shown that it is the partial order induced by the dual cone of $\fa^+$. If we relax the equality in the case $k = d$ to an inequality, we obtain the \emph{weak majorization order} that we denote by $\preceq_w$. We refer to the book \cite{MOA} for more information about majorization.%

In terms of the order $\preceq$, we can formulate \emph{Horn's inequality} (see \cite[Prop.~I.7.4.3]{BLR}) equivalently as%
\begin{equation}\label{eq_horn_inequality}
  \vec{\sigma}(L_1L_2) \preceq \vec{\sigma}(L_1) + \vec{\sigma}(L_2)%
\end{equation}
for any linear operators $L_1$ and $L_2$ that can be composed in the suggested way.%

Now, consider again our dynamical system $(M,f)$. Given \emph{any} Riemannian metric $g$ on $M$, the linear operator $\rmd f^n_x:T_xM \rightarrow T_{f^nx}M$ has unique singular values with respect to the inner products $g_x$ and $g_{f^nx}$. To highlight the dependence of these singular values on the metric $g$, we write $\vec{\sigma}^g(\rmd f^n_x) = \vec{\sigma}(\rmd f^n_x)$. By the MET, for $\mu$-almost every $x \in M$, the \emph{Lyapunov vector}%
\begin{equation*}
  \vec{\lambda}(x) := \lim_{n \rightarrow \infty}\frac{1}{n}\vec{\sigma}^g(\rmd f^n_x) \in \fa^+%
\end{equation*}
exists as a limit and equals the corresponding vector of Lyapunov exponents:%
\begin{equation*}
  \vec{\lambda}(x) = (\lambda_1(x),\lambda_2(x),\ldots,\lambda_d(x)).%
\end{equation*}
Observe that the almost-everyhwere existence of $\vec{\lambda}(x)$ already follows from Kingman's subadditive ergodic theorem, the subadditivity being a consequence of inequality \eqref{eq_horn_inequality}. We also introduce the average over the Lyapunov vector:%
\begin{equation*}
  \vec{\lambda}(f) := \int \vec{\lambda}(x)\, \rmd\mu(x).%
\end{equation*}
Observe that in the ergodic case the integral equals the integrand almost everywhere. As the Lyapunov exponents depend measurably on $x$, and the compactness of $M$ together with the continuity of $\rmd f$ easily implies their boundedness, the integral exists.%

We now prove our first result.%

\begin{proposition}\label{prop_lmin_prop1}
For every $g \in \MC^0$, the following inequality holds:%
\begin{equation*}
  \vec{\lambda}(f) \preceq \int \vec{\sigma}^g(\rmd f_x)\, \rmd\mu(x).%
\end{equation*}
\end{proposition}

\begin{proof}
The inequality to be proven is mainly a consequence of Fekete's subadditivity lemma. We write%
\begin{equation*}
  \lambda_i(f) := \int \lambda_i(x)\, \rmd\mu(x),\quad i = 1,2,\ldots,d.%
\end{equation*}
With this notation at hand, the inequality to be proven is equivalent to%
\begin{equation*}
  \lambda_1(f) + \cdots + \lambda_k(f) \leq \int \sum_{i=1}^k \log\alpha^g_i(\rmd f_x)\, \rmd\mu(x)%
\end{equation*}
for $k = 1,2,\ldots,d$ with equality if $k = d$. For a fixed $k$, consider the sequence%
\begin{equation*}
  a_n := \int \sum_{i=1}^k \log\alpha^g_i(\rmd f^n_x)\, \rmd\mu(x),\quad n = 1,2,\ldots%
\end{equation*}
The sequence is subadditive, which easily follows from Horn's inequality and the $f$-invariance of $\mu$. Hence, Fekete's lemma guarantees that $a_n/n$ converges to $\inf\{a_n/n : n \in \N \}$. This implies%
\begin{align*}
  a_1 &= \int \sum_{i=1}^k \log\alpha^g_i(\rmd f_x)\, \rmd\mu(x) \geq \lim_{n \rightarrow \infty}\frac{1}{n} \int \sum_{i=1}^k \log\alpha^g_i(\rmd f^n_x)\, \rmd\mu(x) \\
			&= \sum_{i=1}^k \int \lim_{n \rightarrow \infty}\frac{1}{n} \log\alpha^g_i(\rmd f^n_x)\, \rmd\mu(x) = \sum_{i=1}^k \int \lambda_i(x)\, \rmd\mu(x) = \sum_{i=1}^k \lambda_i(f).%
\end{align*}
Interchanging the order of limit and integral is possible by the dominated convergence theorem, using that%
\begin{equation}\label{eq_alpha_bounds}
  \min_{z\in M}\log\alpha_d^g(\rmd f_z) \leq \frac{1}{n}\log \alpha^g_i(\rmd f^n_x) \leq \max_{z\in M}\log\alpha_1^g(\rmd f_z)%
\end{equation}
for all $i$, $n$ and $x$. It remains to show the equality in the case $k = d$:%
\begin{align*}
  \sum_{i=1}^d \lambda_i(f) &= \lim_{n \rightarrow \infty}\frac{1}{n}\int \log |\mydet_g \rmd f^n_x|\, \rmd\mu(x) \\
	                          &= \lim_{n \rightarrow \infty}\frac{1}{n}\int \log \prod_{i=0}^{n-1} |\mydet_g \rmd f_{f^ix}|\, \rmd\mu(x) \\
													  &= \lim_{n \rightarrow \infty}\frac{1}{n}\sum_{i=0}^{n-1} \int \log|\mydet_g\rmd f_x|\, \rmd\mu(x) \\
														&= \int \log|\mydet_g\rmd f_x|\, \rmd\mu(x) = \int \sum_{i=1}^d \log\alpha^g_i(\rmd f_x)\, \rmd\mu(x).%
\end{align*}
The first of these equalities follows from statement (ii) of the MET, the second is the multiplicativity of the determinant together with the chain rule, the third uses the invariance of the measure $\mu$, the fourth is obvious, and the last one uses that the absolute determinant equals the product of the singular values.
\end{proof}

Our next goal is to prove a converse result showing that $\vec{\lambda}(f)$ can be approximated by vectors of the form $\int \vec{\sigma}^g(\rmd f_x)\, \rmd\mu(x)$. This result involves the construction of Riemannian metrics used to achieve the desired approximation. For this construction, we use geometric features of the space $\SC^+$ of positive-definite matrices. A good introduction to positive-definite matrices and the geometry of $\SC^+$ is the book \cite{Bha}. In the following, we present some basic facts.%

The space of symmetric $d \tm d$ matrices is $\SC_d := \{ s \in \R^{d \tm d} : s = s\trn \}$. It is a real vector space of dimension $\frac{d}{2}(d+1)$ and contains the open cone of positive-definite matrices $\SC_d^+ := \{ p \in \SC_d : p > 0 \}$. As an open subset, $\SC_d^+$ is a trivial smooth manifold of the same dimension as $\SC_d$. Its tangent space at any point $p \in \SC_d^+$ can be identified canonically with $\SC_d$. A very popular Riemannian metric on $\SC_d^+$, useful in many different contexts, is the \emph{trace metric}, given by%
\begin{equation*}
  \langle v,w \rangle_p := \tr(p^{-1}vp^{-1}w) \mbox{\quad for all\ } p \in \SC_d^+,\ v,w \in T_p\SC^+_d = \SC_d.%
\end{equation*}
Equipped with this metric, $\SC_d^+$ becomes a complete Riemannian manifold. By \cite[Form.~(6.14)]{Bha}, an explicit expression for the induced distance function is%
\begin{equation}\label{eq_splus_distance}
  d(p,q) = \Bigl(\sum_{i=1}^d \log^2\rho_i(p^{-1}q)\Bigr)^{\frac{1}{2}},%
\end{equation}
where $\rho_i(p^{-1}q)$ are the eigenvalues of $p^{-1}q$ (which is similar to $p^{-\frac{1}{2}}qp^{-\frac{1}{2}} \in \SC_d^+$, implying that the eigenvalues are real and positive).%

For each two points $p,q \in \SC^+_d$, there is a unique geodesic segment joining $p$ and $q$, denoted by $p \#_t\, q$ and parameterized on $[0,1]$. An explicit formula is%
\begin{equation*}
  p \#_t\, q = p^{\frac{1}{2}}\bigl(p^{-\frac{1}{2}}qp^{-\frac{1}{2}}\bigr)^t p^{\frac{1}{2}}, \quad t \in [0,1].%
\end{equation*}
If we instead specify a geodesic in $\SC^+_d$ by its initial point $p$ and its initial direction $v$, an explicit formula is%
\begin{equation*}
  \gamma_{p,v}(t) = p^{\frac{1}{2}}\exp(tp^{-\frac{1}{2}}vp^{-\frac{1}{2}})p^{\frac{1}{2}},\quad t \in \R.%
\end{equation*}
The group $\GL(d,\R)$ acts transitively on $\SC^+_d$ by isometries via%
\begin{equation*}
  a \ast p := apa\trn,\quad a \in \GL(d,\R),\ p \in \SC^+_d.%
\end{equation*}
Another isometry of $\SC^+_d$ is the matrix inversion. Using the fact that isometries map geodesics to geodesics, the following identities become obvious:%
\begin{align}\label{eq_splus_isometries_geodesics}
  a \ast (p \#_t\, q) = (a \ast p) \#_t\, (a \ast q) \mbox{\quad and \quad} (p \#_t\, q)^{-1} = p^{-1} \#_t\, q^{-1}.%
\end{align}
Another useful concept for the analysis on $\SC^+_d$ is a vectorial distance given by\footnote{Caution: There is an overload of notation here. We use the letter $d$ for the dimension of $M$, the distance function on $M$ and the vectorial distance on $\SC^+_d$. It is hopefully always clear from the context what is meant.}%
\begin{equation*}
  \vec{d}(p,q) := 2\vec{\sigma}(p^{-\frac{1}{2}}q^{\frac{1}{2}}),\quad p,q \in \SC^+_d.%
\end{equation*}

The vectorial distance is closely related to the distance \eqref{eq_splus_distance} induced by the trace metric. The next proposition summarizes some of its properties, from which we will only use (i) and (ii). See also \cite[Sec.~4]{Boc}.%

\begin{proposition}\label{prop_vecd}
The vectorial distance $\vec{d}(\cdot,\cdot):\SC^+_d \tm \SC^+_d \rightarrow \fa^+$ has the following properties:%
\begin{enumerate}
\item[(i)] $\vec{d}(I,p) = \vec{\sigma}(p)$ and $\|\vec{d}(p,q)\|_2 = d(p,q)$ for all $p,q \in \SC^+_d$.%
\item[(ii)] $\vec{d}(p_1,q_1) = \vec{d}(p_2,q_2)$ if and only if there exists $a \in \GL(d,\R)$ with $a \ast p_1 = p_2$ and $a \ast q_1 = q_2$.%
\item[(iii)] $\vec{d}(p,q) \preceq \vec{d}(p,r) + \vec{d}(r,q)$ for all $p,q,r \in \SC^+_d$.%
\item[(iv)] $\vec{d}(q,p) = i(\vec{d}(p,q))$ with $i(\xi) := -(\xi_d,\xi_{d-1},\ldots,\xi_1)$ for all $p,q \in \SC^+_d$.%
\end{enumerate}
\end{proposition}

Another important ingredient for the proof of the converse of Proposition \ref{prop_lmin_prop1} is the following concept. The \emph{barycenter} of $l$ matrices $p_1,\ldots,p_l \in \SC^+_d$ ($l \in \N$) is defined as%
\begin{equation*}
  \bary(p_1,\ldots,p_l) := \argmin_{p \in \SC^+_d}\sum_{i=1}^l d(p,p_i)^2.%
\end{equation*}
It can be shown that the minimizer exists and is unique, so $\bary(p_1,\ldots,p_l)$ is well-defined. The following properties are well-known, see, e.g., \cite{LLi}:%
\begin{itemize}
\item Symmetry: $\bary(p_1,\ldots,p_l) = \bary(p_{\sigma(1)},\ldots,p_{\sigma(l)})$ for any permutation $\sigma$.%
\item For any $a \in \GL(d,\R)$,%
\begin{equation*}
  a \ast \bary(p_1,\ldots,p_l) = \bary(a \ast p_1,\ldots,a \ast p_l).%
\end{equation*}
\item If $u = \bary(p_1,\ldots,p_{l-1},p_l)$ and $v = \bary(p_1,\ldots,p_{l-1},p_l')$, then%
\begin{equation}\label{eq_vecd_ineq}
  \vec{d}(u,v) \preceq \frac{1}{l}\vec{d}(p_l,p_l').%
\end{equation}
\item The mapping $(p_1,\ldots,p_l) \mapsto \bary(p_1,\ldots,p_l)$ is continuous.%
\end{itemize}

We need to lift some of the constructions introduced above from $\SC^+_d$ to the space $\MC^0$ of continuous metrics on $M$. If $F:M \hookleftarrow$ is a $C^1$-diffeomorphism and $g \in \MC^0$, then $F$ acts on $g$ by pulling back:%
\begin{equation*}
  (F^*g)_x(v,w) := g_{Fx}(\rmd F_xv, \rmd F_xw) \mbox{\quad for all\ } v,w \in T_xM.%
\end{equation*}
Since $F$ is $C^1$, we have $F^*g \in \MC^0$. We will see below that this can be regarded as an extension of the action of $\GL(d,\R)$ on $\SC^+_d$. We also note that for iterates of $F$ we write $F^{n*}g$ instead of the more cumbersome $(F^n)^*g$.%

Next, we extend the barycenter map to Riemannian metrics. Given $l$ metrics $g^1,\ldots,g^l$ in $\MC^0$, we can define a new metric, denoted by $\bary(g^1,\ldots,g^l)$, and defined pointwise via the barycenter of positive-definite matrices. This is done as follows. Let $(\phi,U)$ be any chart of $M$ around $x \in M$, i.e.~$U \subset M$ is an open neighborhood of $x$ and $\phi:U \rightarrow \phi(U) \subset \R^d$ a  diffeomorphism. Then we identify an element $p \in \SC^+_d$ with an inner product on $T_xM$ via\footnote{This identification is unusual because of the involved inversion of $p$. However, for our purposes it is more convenient as it allows to obtain more direct relations between the introduced operations on $\SC^+_d$ and corresponding operations on $\MC^0$.}%
\begin{equation}\label{eq_metric_matrix_rep}
  \beta_{\phi}(p)(v,w) := \langle p^{-1} \rmd\phi_xv,\rmd\phi_xw \rangle \mbox{\quad for all\ } v,w \in T_xM,%
\end{equation}
where $\langle\cdot,\cdot\rangle$ denotes the standard Euclidean inner product on $\R^d$. The map $\beta_{\phi}$ is obviously invertible. If $(\phi_1,U)$ and $(\phi_2,V)$ are two charts around $x$, then the two matrix representations of the inner product $g_x$ are related by%
\begin{equation}\label{eq_localrep_rel}
  \rmd(\phi_2 \circ \phi_1^{-1})_{\phi_1(x)} \ast \beta_{\phi_1}^{-1}(g_x) = \beta_{\phi_2}^{-1}(g_x),%
\end{equation}
which is an easy consequence of the definition in \eqref{eq_metric_matrix_rep}. From this and Proposition \ref{prop_vecd}(ii), it follows that the vectorial distance of two inner products $g^1_x$ and $g^2_x$ on $T_xM$ can be defined independently of the chosen chart:%
\begin{equation*}
  \vec{d}(g^1_x,g^2_x) := \vec{d}(\beta_{\phi}^{-1}(g^1_x),\beta_{\phi}^{-1}(g^2_x)).%
\end{equation*}
The barycenter of $l$ metrics $g^1,\ldots,g^l$ is now defined by%
\begin{equation}\label{eq_def_bary_for_metrics}
  \bary(g^1,\ldots,g^l)_x := \beta_{\phi}(\bary(\beta_{\phi}^{-1}(g^1_x),\ldots,\beta_{\phi}^{-1}(g^l_x))),%
\end{equation}
where for each $x \in M$ we choose an appropriate chart $(\phi,U)$ around $x$. Using \eqref{eq_localrep_rel}, it can easily be shown that this definition is independent of the chosen chart. From \eqref{eq_def_bary_for_metrics}, it follows that $\bary(g^1,\ldots,g^l)$ is continuous in the domain of every chart, since all involved functions on the right-hand side are continuous. Hence, $\bary(g^1,\ldots,g^l) \in \MC^0$.%

The proof of the following lemma is trivial, since all statements follow from a direct application of the definitions. Hence, we leave its proof to the reader.%

\begin{lemma}\label{lem_pullback}
For all $g,g^i \in \MC^0$ and $C^1$-diffeomorphisms $F,F_i:M \hookleftarrow$, the following properties hold:%
\begin{enumerate}
\item[(i)] $F^*\bary(g^1,\ldots,g^l) = \bary(F^*g^1,\ldots,F^*g^l)$.%
\item[(ii)] $F_2^* (F_1^*g) = (F_1 \circ F_2)^* g$.%
\item[(iii)] Let $(\phi,U)$ be a chart around $x$ and $(\psi,V)$ one around $Fx$. Then%
\begin{equation*}
  \beta_{\phi}^{-1}((F^*g)_x) = \rmd(\psi \circ F \circ \phi^{-1})_{\phi(x)}^{-1} \ast \beta_{\psi}^{-1}(g_{Fx}).%
\end{equation*}
In this sense, the pullback of a metric via a diffeomorphism is an extension of the action of $\GL(d,\R)$ on $\SC^+_d$.%
\item[(iv)] $\vec{d}(g^1_{Fx},g^2_{Fx}) = \vec{d}((F^* g^1)_x,(F^*g^2)_x)$ for all $x\in M$.
\end{enumerate} 
\end{lemma}

\begin{lemma}\label{lem_singvals}
Let $F:M \hookleftarrow$ be a $C^1$-diffeomorphism and $g \in \MC^0$. Then, for any $x \in M$ we have%
\begin{equation*}
  \vec{\sigma}^g(\rmd F_x) = \vec{\sigma}(p_2^{-\frac{1}{2}} A_x p_1^{\frac{1}{2}}),%
\end{equation*}
where $p_1 = \beta_{\phi}^{-1}(g_x)$, $p_2 = \beta_{\psi}^{-1}(g_{Fx})$ and $A_x = \rmd(\psi \circ F \circ \phi^{-1})_{\phi(x)}$ for charts $(\phi,U)$ and $(\psi,V)$ around $x$ and $Fx$, respectively.
\end{lemma}

\begin{proof}
The singular values of $\rmd F_x:T_xM \rightarrow T_{Fx}M$ with respect to the metric $g$ are the square roots of the eigenvalues of the self-adjoint operator%
\begin{equation*}
  (\rmd F_x)(\rmd F_x)^*:T_{Fx}M \rightarrow T_{Fx}M.%
\end{equation*}
Using the charts $(\phi,U)$ and $(\psi,V)$, we have for any $v \in T_xM$, $w \in T_{Fx}M$ that%
\begin{align*}
  &g_{Fx}(\rmd F_x v, w) = \langle \beta_{\psi}^{-1}(g_{Fx})^{-1} \rmd\psi_{Fx}\rmd F_x v, \rmd\psi_{Fx}w \rangle \\
	                      &= \langle \rmd\psi_{Fx}\rmd F_x v, [\beta_{\psi}^{-1}(g_{Fx})^{-1}]\trn \rmd\psi_{Fx}w \rangle \\
												&= \langle \rmd\phi_x v, [\rmd\psi_{Fx}\rmd F_x (\rmd\phi_x)^{-1}]\trn [\beta_{\psi}^{-1}(g_{Fx})^{-1}]\trn \rmd\psi_{Fx} w \rangle \\
      									&= \langle \beta_{\phi}^{-1}(g_x)^{-1} \rmd\phi_x v, \beta_{\phi}^{-1}(g_x) [\rmd\psi_{Fx}\rmd F_x (\rmd\phi_x)^{-1}]\trn \beta_{\psi}^{-1}(g_{Fx})^{-1} \rmd\psi_{Fx} w \rangle.%
\end{align*}
In the last identity, we used, in particular, that the matrix $\beta_{\phi}^{-1}(g_x)$ is symmetric. On the other hand,%
\begin{align*}
  g_x(v,(\rmd F_x)^* w) = \langle \beta_{\phi}^{-1}(g_x)^{-1} \rmd\phi_x v, \rmd\phi_x (\rmd F_x)^* w \rangle.%
\end{align*}
When we compare the two results, we see that%
\begin{equation*}
  (\rmd F_x)^* = \rmd\phi_x^{-1}\beta_{\phi}^{-1}(g_x) [\rmd\psi_{Fx}\rmd F_x (\rmd\phi_x)^{-1}]\trn \beta_{\psi}^{-1}(g_{Fx})^{-1} \rmd\psi_{Fx}.%
\end{equation*}
In the middle of the right-hand side term, we find the transpose of the local representation of $\rmd F_x$, that we denote by $A_x$. Then%
\begin{align*}
  \rmd F_x (\rmd F_x)^* = \rmd\psi_{Fx}^{-1} A_x \beta_{\phi}^{-1}(g_x) A_x\trn \beta_{\psi}^{-1}(g_{Fx})^{-1} \rmd\psi_{Fx}.%
\end{align*}
It follows that the singular values of $\rmd F_x$ with respect to the metric $g$ are the square roots of the eigenvalues of the matrix $C_x := A_x \beta_{\phi}^{-1}(g_x) A_x\trn \beta_{\psi}^{-1}(g_{Fx})^{-1}$ which is similar to%
\begin{equation*}
  [p_2^{-\frac{1}{2}} A_x p_1^{\frac{1}{2}}] [p_1^{\frac{1}{2}} A_x\trn p_2^{-\frac{1}{2}}] = B_xB_x\trn, \quad B_x := p_2^{-\frac{1}{2}} A_x p_1^{\frac{1}{2}}.%
\end{equation*}
It follows that the singular values of $\rmd F_x$ in the metric $g$ are the eigenvalues of $(B_xB_x\trn)^{\frac{1}{2}}$, or equivalently, the ordinary singular values of $B_x$, which directly implies the statement of the lemma.
\end{proof}

Now, we can finally prove our converse result.%

\begin{proposition}\label{prop_lmin_prop2}
For any $\ep>0$, there exists a metric $g^{\ep} \in \MC^0$ such that%
\begin{equation*}
  \int \vec{\sigma}^{g^{\ep}}(\rmd f_x)\, \rmd\mu(x) \preceq_w \vec{\lambda}(f) + \ep\unit.%
\end{equation*}
\end{proposition}

\begin{proof}
Using a fixed reference metric $g^0 \in \MC^0$, for each $N \in \N$ we define the new metric%
\begin{equation*}
  g^N := \bary(g^0,f^*g^0,f^{2*}g^0,\ldots,f^{(N-1)*}g^0)%
\end{equation*}
via the iterates of $f$ (in particular, $g^1 = g^0$). Then, using Lemma \ref{lem_pullback}(i) and (ii) together with the symmetry of the barycenter map, we find that%
\begin{align*}
  f^{(-1)*}g^N &= \bary(f^{(-1)*}g^0, g^0,\ldots,f^{(N-2)*}g^0) \\
	             &= \bary(g^0, f^*g^0, \ldots, f^{(N-2)*}g^0, f^{(-1)*}g^0).%
\end{align*}
We now compare the two metrics $g^N$ and $f^{(-1)*}g^N$ at the point $fx$ (for an arbitrary $x \in M$), using inequality \eqref{eq_vecd_ineq}:%
\begin{align*}
 & \vec{d}(g^N_{fx}, (f^{(-1)*}g^N)_{fx}) \\
 &= \vec{d}( \bary(\beta_{\phi}^{-1}(g^0_{fx}),\beta_{\phi}^{-1}([f^*g^0]_{fx}),\ldots,\beta_{\phi}^{-1}([f^{(N-1)*}g^0]_{fx})), \\
  &\qquad\qquad \bary(\beta_{\phi}^{-1}(g^0_{fx}),\beta_{\phi}^{-1}([f^*g^0]_{fx}),\ldots,\beta_{\phi}^{-1}([f^{(-1)*}g^0]_{fx}))) \\
 &\preceq \frac{1}{N}\vec{d}(\beta_{\phi}^{-1}([f^{(N-1)*}g^0]_{fx}), \beta_{\phi}^{-1}([f^{(-1)*}g^0]_{fx}) ) \\
 &= \frac{1}{N}\vec{d}(( f^{(N-1)*} g^0)_{fx},(f^{(-1)*}g^0)_{fx}) = \frac{1}{N}\vec{d}(g^0_{f^Nx},( f^{(-N)*}g^0)_{f^Nx}).%
\end{align*}
For the last identity, we used Lemma \ref{lem_pullback}(iv). Choose charts $(\phi,U)$ and $(\psi,V)$ around the points $f^Nx$ and $x$, respectively. Let $p(f^Nx)$ be the matrix representation of $g^0$ at $f^Nx$ and $p(x)$ the matrix representation of $g^0$ at $x$. Then, using Lemma \ref{lem_pullback}(iii), we find that%
\begin{align*}
  \vec{d}(g^0_{f^Nx},( f^{(-N)*} g^0)_{f^Nx}) = \vec{d}( p(f^Nx), \rmd (\psi \circ f^{-N} \circ \phi^{-1} )^{-1}_{\phi(f^Nx)} \ast p(x) ).%
\end{align*}
We observe that $\rmd(\psi \circ f^{-N} \circ \phi^{-1})_{\phi(f^Nx)}$ is the inverse of the local representation of $\rmd f^N$ at $x$, for which we write $A^N_x$. Hence, using Proposition \ref{prop_vecd}(ii), we obtain%
\begin{align*}
  \vec{d}(g^0_{f^Nx},(f^{(-N)*} g^0)_{f^Nx}) &= \vec{d}( p(f^Nx), A^N_x \ast p(x) ) \\
	                                           &= \vec{d}( I, [p(f^Nx)^{-\frac{1}{2}} A^N_x p(x)^{\frac{1}{2}}][p(f^Nx)^{-\frac{1}{2}} A^N_x p(x)^{\frac{1}{2}}]\trn ).%
\end{align*}
Writing $B^N_x := p(f^Nx)^{-\frac{1}{2}} A^N_x p(x)^{\frac{1}{2}}$, Proposition \ref{prop_vecd}(i) and Lemma \ref{lem_singvals} yield%
\begin{align*}
  \vec{d}(g^0_{f^Nx},(f^{(-N)*} g^0)_{f^Nx}) &= \vec{\sigma}(B^N_x(B^N_x)\trn) = 2\vec{\sigma}(B^N_x) \\
	&= 2\vec{\sigma}(p(f^Nx)^{-\frac{1}{2}}A^N_xp(x)^{\frac{1}{2}}) = 2\vec{\sigma}^{g^0}(\rmd f^N_x).%
\end{align*}
Using the same arguments for $\vec{d}(g^N_{fx}, (f^{(-1)*}g^N)_{fx})$, we have thus proven that%
\begin{equation}\label{eq_intermediate_statement}
  \vec{\sigma}^{g^N}(\rmd f_x) \preceq \frac{1}{N}\vec{\sigma}^{g^0}(\rmd f^N_x) \mbox{\quad for all\ } x \in M.%
\end{equation}
To prove the statement of the proposition, observe that%
\begin{equation*}
  \vec{\lambda}(f) = \int \lim_{n \rightarrow \infty}\frac{1}{n}\vec{\sigma}^{g^0}(\rmd f^n_x)\, \rmd\mu(x) = \lim_{n \rightarrow \infty}\int \frac{1}{n}\vec{\sigma}^{g^0}(\rmd f^n_x)\, \rmd\mu(x),%
\end{equation*}
where we use again the dominated convergence theorem. Hence, for the given $\ep > 0$, we can find $N$ such that%
\begin{equation*}
  \int \vec{\sigma}^{g^N}(\rmd f_x)\, \rmd\mu(x) \stackrel{\eqref{eq_intermediate_statement}}{\preceq} \int \frac{1}{N}\vec{\sigma}^{g^0}(\rmd f^N_x)\, \rmd\mu(x) \preceq_w \vec{\lambda}(f) + \ep\unit,%
\end{equation*}
which completes the proof.
\end{proof}

\begin{remark}
A careful inspection of the proof shows that the assumption of invertibility of $f$ is actually not needed. It suffices to assume that the derivative $\rmd f_x$ is an invertible linear operator for each $x \in M$. This is because, even though we are pulling back the metric $g^N$ by $f^{-1}$, we are only evaluating the resulting metric at $fx$, and it holds that%
\begin{equation*}
  (f^{(-1)*} g)_{fx} = g_x(\rmd f^{-1}_{fx}v,\rmd f^{-1}_{fx}w) = g_x((\rmd f_x)^{-1}v,(\rmd f_x)^{-1}w).%
\end{equation*}
Also in \cite{Boc}, only the invertibility of the values of the linear cocycle is required. 
\end{remark}

Proposition \ref{prop_lmin_prop1} and \ref{prop_lmin_prop2} together show that the Lyapunov vector $\vec{\lambda}(f)$ can be regarded as a \emph{strong infimum}\footnote{We use the term \emph{strong infimum} in contrast to a \emph{weak ($=$ Pareto) infimum}.} of the vector-valued function%
\begin{equation*}
  \vec{\SC}_{f,\mu}:g \mapsto \int\vec{\sigma}^g(\rmd f_x)\, \rmd\mu(x), \quad \vec{\SC}_{f,\mu}:\MC^0 \rightarrow \fa^+.%
\end{equation*}
The infimum needs to be understood with respect to the order $\preceq$ that was defined on the convex cone $\fa^+$. Of course, the space $\MC^0$ is huge, so it would be convenient if we could reduce this minimization problem to a smaller subspace. A first observation is that the multiplication with a positive scalar function does not change the value of $\vec{\SC}_{f,\mu}$:%
\begin{equation*}
  \vec{\SC}_{f,\mu}(\gamma \cdot g) = \vec{\SC}_{f,\mu}(g) \mbox{\quad for all\ } \gamma \in C^0(M,\R_{>0}),\ g \in \MC^0.%
\end{equation*}
Indeed, Lemma \ref{lem_singvals} implies%
\begin{equation*}
  \log\alpha_i^{\gamma \cdot g}(\rmd f_x) = \frac{1}{2}(\log \gamma(fx) - \log \gamma(x)) + \log \alpha_i^g(\rmd f_x),%
\end{equation*}
and the integral over $\log\gamma(fx) - \log\gamma(x)$ vanishes because $\mu$ is $f$-invariant. It follows that we can restrict the search for an optimal metric to the space of metrics inducing the same (arbitrary) volume form on $M$. To see this, let a volume form $\omega_0$ be fixed and pick an arbitrary $g \in \MC^0$. The metric $g$ induces a volume form $\omega_g = \sqrt{\det g}\,\rmd x^1 \wedge \ldots \wedge \rmd x^d$. Then, there exists a positive function $\gamma$ such that $\omega_g = \gamma \cdot \omega_0$, and the metric $\tilde{g} := \gamma^{-2/d} \cdot g$ induces the volume form $\gamma^{-1} \cdot \omega_g = \omega_0$. In particular, in dimension $d = 1$ there is no need to optimize, because here any metric is just a positive scalar function.%

Furthermore, we can restrict $\vec{\SC}_{f,\mu}$ to the space of smooth metrics on $M$, since every continuous metric can be approximated uniformly (by compactness of $M$) by smooth metrics and the singular values depend continuously on the metric in the uniform topology.%

Recall that $\MC$ denotes the space of all smooth Riemannian metrics on $M$. Given a smooth volume form $\omega$, we let $\MC_{\omega} \subset \MC$ denote the subspace of metrics inducing $\omega$. We then have the following theorem which will be the basis of our further investigations.%

\begin{theorem}\label{thm_bochi}
Let $f:M \hookleftarrow$ be a $C^1$-diffeomorphism on a compact Riemannian manifold, preserving the probability measure $\mu$. Then, for every smooth volume form $\omega$ on $M$, the following statements hold:%
\begin{enumerate}
\item[(i)] For every $g \in \MC_{\omega}$, it holds that%
\begin{equation*}
  \vec{\lambda}(f) \preceq \vec{\SC}_{f,\mu}(g).%
\end{equation*}
\item[(ii)] For every $\ep > 0$, there exists $g^{\ep} \in \MC_{\omega}$ such that%
\begin{equation*}
  \vec{\SC}_{f,\mu}(g^{\ep}) \preceq_w \vec{\lambda}(f) + \ep\unit.%
\end{equation*}
\item[(iii)] For every $\ep > 0$, there exists $g^{\ep} \in \MC_{\omega}$ such that\footnote{Statement (iii) and its proof have been suggested by one of the reviewers.}%
\begin{equation*}
  \|\vec{\sigma}^{g^{\ep}}(\rmd f_{\cdot}) - \vec{\lambda}(\cdot)\|_1 = \int \|\vec{\sigma}^{g^{\ep}}(\rmd f_x) - \vec{\lambda}(x)\|\, \rmd\mu(x) \leq \ep.%
\end{equation*}
\end{enumerate}
\end{theorem}

\begin{proof}
It only remains to prove (iii). To this end, let us first recap the following simple fact from measure theory:%

\emph{Let $(\varphi_n)$ and $(\psi_n)$ be two sequences of real-valued functions in $L^1(\mu)$ satisfying $\varphi_n \leq \psi_n$ for all $n$. Further suppose that $\psi_n$ converges in $L^1(\mu)$ to some function $\psi$ and $\int \psi\, \rmd\mu \leq \int \varphi_n\, \rmd\mu$ for all $n$. Then $\varphi_n$ also converges to $\psi$ in $L^1(\mu)$.}%

This follows from the estimate%
\begin{align*}
  |\varphi_n(x) - \psi(x)| + (\varphi_n(x) - \psi(x)) &= 2(\varphi_n(x) - \psi(x))^+ \\
	&\leq 2(\psi_n(x) - \psi(x))^+ \leq 2|\psi_n(x) - \psi(x)|,%
\end{align*}
which implies $\|\varphi_n - \psi\|_1 \leq 2\|\psi_n - \psi\|_1$.%

Now, consider the pointwise inequality%
\begin{equation*}
  \vec{\sigma}^{g^N}(\rmd f_x) \preceq \frac{1}{N}\vec{\sigma}^{g^0}(\rmd f^N_x)%
\end{equation*}
taken from \eqref{eq_intermediate_statement}. According to the MET, the right-hand side converges in $L^1(\mu)$ to $\vec{\lambda}(x)$ as $N \rightarrow \infty$. On the other hand, Proposition \ref{prop_lmin_prop1} gives%
\begin{equation*}
  \int \vec{\lambda}(x) \rmd\mu(x) \preceq \int \vec{\sigma}^{g^N}(\rmd f_x)\, \rmd\mu(x).%
\end{equation*}
We can now apply the fact proven above componentwise to obtain statement (iii).
\end{proof}

\begin{remark}
It is a simple exercise to show that the inequality%
\begin{equation}\label{eq_sandwich}
  a \preceq_w b \preceq_w a + \ep\unit%
\end{equation}
for vectors $a,b \in \fa^+$ implies $\|a - b\|_{\infty} \leq \ep$. Hence, Theorem \ref{thm_bochi} (i) and (ii) together imply that $\vec{\lambda}(f)$ can be approximated by vectors of the form $\vec{\SC}_{f,\mu}(g)$ with $g \in \MC_{\omega}$, i.e.%
\begin{equation*}
  \left\| \int (\vec{\lambda}(x) - \vec{\sigma}^g(\rmd f_x))\, \rmd\mu(x) \right\|%
\end{equation*}
can be made arbitrarily small by an appropriate choice of $g$. Statement (iii) of the theorem provides the stronger statement that%
\begin{equation*}
  \int \|\vec{\lambda}(x) - \vec{\sigma}^g(\rmd f_x)\|\, \rmd\mu(x)%
\end{equation*}
can be made arbitrarily small.
\end{remark}

\begin{remark}
It should also be noted that the sequence of metrics employed in the proof of Proposition \ref{prop_lmin_prop2} is not very helpful for concrete computations, because it involves the derivatives of the iterates of $f$ just like the definition of the Lyapunov exponents.
\end{remark}

\section{The space of Riemannian metrics}\label{sec_riem_metrics}

To explore deeper properties of the optimization problem for $g \mapsto \vec{\SC}_{f,\mu}(g)$, we need to gain a better understanding of the spaces $\MC$ and $\MC_{\omega}$ and their natural geometric structures. In this section, we gather some facts from the literature about these spaces, mainly taken from \cite{Ebi,FGr,Cla}. We also refer to \cite{Sch} for a comprehensive general treatment of manifolds of smooth mappings.%

A Riemannian metric $g$ assigns to each $x \in M$ an inner product on the tangent space $T_xM$, i.e.~a positive-definite symmetric $(0,2)$-tensor $g_x:T_xM \tm T_xM \rightarrow \R$. We write $T^0_2M$ for the tensor bundle of $(0,2)$-tensors on $M$ and $S^0_2M \subset T^0_2M$ for the subbundle of symmetric tensors. This subbundle attaches to each $x \in M$ the vector space of all symmetric $(0,2)$-tensors on $T_xM$. In this formal setup, a Riemannian metric is a special section of the bundle $S^0_2M$, namely one that is positive-definite. We write $\SC^0_2M$ for the space of smooth sections of $S^0_2M$, i.e., of symmetric $(0,2)$-tensor fields of class $C^{\infty}$. Obviously, $\SC^0_2M$ is an infinite-dimensional vector space over the reals with the pointwise addition and scalar multiplication. Indeed, $\SC^0_2M$ is a Fr\'echet space with the $C^{\infty}$-topology, of which $\MC$ is an open subset. The fact that $\MC$ is open in the Fr\'echet space $\SC^0_2M$ implies that $\MC$ is a trivial Fr\'echet manifold whose tangent space at any point can be identified canonically with $\SC^2_0M$.%

We define an inner product on each tangent space of $\MC$ as%
\begin{equation*}
  \langle h,k \rangle_g := \int \tr(g^{-1}_xh_xg^{-1}_xk_x)\, \rmd\omega_g(x) \mbox{\quad for all\ } h,k \in T_g\MC = \SC^2_0M,%
\end{equation*}
where $\omega_g$ is the volume form induced by $g$, and $g_x$, $h_x$, $k_x$ have to be understood as local representations. With similar arguments as used in the previous section, one shows that the trace is independent of the charts used to obtain these representations. In terms of the trace metric $\langle \cdot,\cdot \rangle_{g_x}$, we can also write%
\begin{equation*}
  \langle h,k \rangle_g = \int \langle h_x,k_x \rangle_{g_x}\, \rmd\omega_g(x).%
\end{equation*}
For obvious reasons, this Riemannian metric is called the \emph{$L^2$-metric} on $\MC$. It is only a \emph{weak} metric, meaning that the induced topology on the tangent space does not coincide with the inherited topology of the manifold. Nevertheless, as shown in \cite{Cla}, the geodesic distance induced by the $L^2$-metric turns $\MC$ into a (noncomplete) metric space.\footnote{As on a finite-dimensional manifold, the geodesic distance of two points $g^1,g^2$ is defined as the infimum over the lengths of all piecewise differentiable curves connecting $g^1$ and $g^2$. On general weak Riemannian manifolds, this is only a pseudo-metric.} It can also be seen easily (using Lemma \ref{lem_pullback}(iii)) that the $L^2$-metric is invariant under the diffeomorphism group of $M$ acting by pullback. An explicit expression for the distance function on $\MC$, induced by the $L^2$-metric, was given in \cite{Cl2} (see also \cite[App.~B]{CSu} for a more elegant proof):%
\begin{equation}\label{eq_l2_distance}
  d_{L^2}(g^1,g^2) = \Bigl(\int_M d_x(g^1_x,g^2_x)^2\, \rmd\omega_{g^0}(x)\Bigr)^{\frac{1}{2}},%
\end{equation}
where $g^0$ is an arbitrary reference metric with $\vol(M,g^0) = 1$, and $d_x(\cdot,\cdot)$ is the fiber metric on the space of inner products on $T_xM$, induced by the Riemannian metric $\langle a,b \rangle_p = \tr(p^{-1}ap^{-1}b)\sqrt{\det(g^0(x)^{-1}p)}$.%

The space $\MC$ admits the global product structure $\MC \cong \VC \tm \MC_{\omega}$, where $\VC$ denotes the space of smooth volume forms on $M$. The space $\VC$ is a Fr\'echet manifold, in fact an open subset of $\Omega^d(M)$, the Fr\'echet space of order-$d$ differential forms on $M$. Given $\omega \in \VC$ and $\nu \in \Omega^d(M)$, there exists a unique smooth function, denoted by $\frac{\nu}{\omega}$, such that%
\begin{equation*}
  \nu = \left(\frac{\nu}{\omega}\right)\omega.%
\end{equation*}
In case that $\nu$ is also a volume form, the function $\frac{\nu}{\omega}$ is strictly positive. The space $\MC_{\omega}$ is a smooth submanifold of $\MC$ with tangent space%
\begin{equation*}
  T_g\MC_{\omega} = \{ h \in \SC^0_2M : \tr(g_x^{-1}h_x) = 0 \mbox{\ for all\ } x \in M \}.%
\end{equation*}

The following result is proven in \cite[Prop.~1.13 and Prop.~2.2]{FGr}.\footnote{Actually, the result says much more, namely that $\MC_{\omega}$ is a symmetric space.}%

\begin{proposition}\label{prop_mcnu}
For each $\omega \in \VC$, the submanifold $\MC_{\omega}$ is geodesically complete and the geodesic starting at $g^0 \in \MC_{\omega}$ with initial tangent $h$ is given by%
\begin{equation}\label{eq_geodesic_formula}
  g^t = g^0 \exp(t (g^0)^{-1}h).%
\end{equation}
\end{proposition}

\begin{remark}
The formula \eqref{eq_geodesic_formula} needs to be understood pointwise. That is,%
\begin{equation*}
  g^t_x = g^0_x \exp(t (g^0_x)^{-1} h_x) \mbox{\quad for all\ } x \in M,%
\end{equation*}
where we can use local representations of the involved objects $g^0$ and $h$ with respect to a fixed chart, and $\exp(\cdot)$ is the usual matrix exponential.
\end{remark}

\begin{remark}
The statement that $\MC_{\omega}$ is geodesically complete does not imply that $\MC_{\omega}$ is complete as a metric space, when equipped with the geodesic distance. In finite dimensions, this implication holds by the Hopf-Rinow theorem, but in infinite dimensions there are counter-examples. To the best of my knowledge, it is not known if $\MC_{\omega}$ is also complete as a metric space (but most probably it is not, because $\MC$ is not complete).%
\end{remark}

\begin{remark}
The submanifold $\MC_{\omega}$ is not totally geodesic in $\MC$. Indeed, the formula for geodesics in $\MC$ is different, and will not be used in this paper.
\end{remark}

\begin{corollary}\label{cor_geodesics}
For each pair of metrics $g^a,g^b \in \MC_{\omega}$, there exists a unique geodesic segment on $[0,1]$, connecting $g^a$ and $g^b$. It is given by%
\begin{equation}\label{eq_geodesic_endpoints_formula}
  g^t_x := (g^a_x \#_t\, g^b_x) \mbox{\quad for all\ } x \in M,\ t \in [0,1],%
\end{equation}
where the operation $\#_t$ carries over from $\SC^+_d$ to the space of symmetric positive-definite $(0,2)$-tensors on $T_xM$ via local representations.
\end{corollary}

\begin{proof}
The proof consists of two parts. First, we show that the curve given by \eqref{eq_geodesic_endpoints_formula} is a geodesic segment in $\MC_{\omega}$ which connects $g^a$ and $g^b$. Second, we show that each such geodesic segment is of the given form.%

(1) The curve $t \mapsto g^t$ is well-defined as a curve in $\MC$ because of \eqref{eq_localrep_rel} and \eqref{eq_splus_isometries_geodesics}. Obviously, $g^0 = g^a$ and $g^1 = g^b$. To show that $t \mapsto g^t$ is a geodesic segment in $\MC_{\omega}$, we define (in terms of local representations)%
\begin{equation*}
  h_x := (g^a_x)^{\frac{1}{2}}\log((g^a_x)^{-\frac{1}{2}} g^b_x (g^a_x)^{-\frac{1}{2}})(g^a_x)^{\frac{1}{2}} \mbox{\quad for all\ } x \in M.%
\end{equation*}
Here, $\log(\cdot)$ denotes the matrix logarithm, which is well-defined and unique as an operator on $\SC^+_d$. It is clear that $h_x$ is symmetric for each $x$ and easy to show that $h_x$ is well-defined, i.e.~independent of the representation. Hence, $h$ is a section of $S^0_2M$. It is a smooth section, because it is constructed from smooth sections via smooth operations. Hence, $h \in \SC^0_2M$. Finally, we have%
\begin{align*}
  \tr((g^1_x)^{-1}h_x) &= \tr(\log((g^a_x)^{-\frac{1}{2}} g^b_x (g^a_x)^{-\frac{1}{2}})) = \log \det ((g^a_x)^{-\frac{1}{2}} g^b_x (g^a_x)^{-\frac{1}{2}})) \\
	&= \log[\det(g^a_x)^{-1}\det(g^b_x)] = \log(1) = 0,%
\end{align*}
because the fact that $g^a$ and $g^b$ induce the same volume form implies $\det(g^a_x) = \det(g^b_x)$ for all $x \in M$. This shows that $h \in T_{g^a}\MC_{\omega}$. To see that $t \mapsto g^t$ is a geodesic segment, it remains to show that $g^t = g^a \exp(t (g^a)^{-1}h)$ and invoke Proposition \ref{prop_mcnu}. This, however, is a simple computation, and we leave it to the reader.%

(2) To show the uniqueness, let $t \mapsto \tilde{g}^t$ be any geodesic segment in $\MC_{\omega}$ with $\tilde{g}^0 = g^a$ and $\tilde{g}^1 = g^b$. By Proposition \ref{prop_mcnu}, it must be of the form $\tilde{g}^t = g^a \exp(t (g^a)^{-1}h)$ for some $h \in T_{g^a}\MC_{\omega}$, implying $g^a_x\exp( (g^a_x)^{-1}h_x) = g^b_x$ for all $x \in M$. Thus,%
\begin{align*}
  h_x &= (g^a_x)\log((g^a_x)^{-1}g^b_x) = (g^a_x)\log((g^a_x)^{-\frac{1}{2}}(g^a_x)^{-\frac{1}{2}}g^b_x (g^a_x)^{-\frac{1}{2}}(g^a_x)^{\frac{1}{2}}) \\
			&= (g^a_x)^{\frac{1}{2}}\log((g^a_x)^{-\frac{1}{2}} g^b_x (g^a_x)^{-\frac{1}{2}})(g^a_x)^{\frac{1}{2}}.%
\end{align*}
Hence, by the first part of the proof, it follows that $\tilde{g}^t \equiv g^t$.
\end{proof}

\section{Convexity}\label{sec_convexity}

In this section, we prove that the function $\vec{\SC}_{f,\mu}:\MC_{\omega} \rightarrow \fa^+$ is convex with respect to the structures introduced on its domain and co-domain. First, we need to prove continuity, which is implied by the following lemma.%

\begin{lemma}\label{lem_continuity}
The function $(g,x) \mapsto \vec{\sigma}^g(\rmd f_x)$ is continuous on $\MC \tm M$.
\end{lemma}

\begin{proof}
We denote the function under consideration by $\beta$. By Lemma \ref{lem_singvals}, we can write $\beta$ locally as%
\begin{equation*}
  \beta(g,x) = \vec{\sigma}( \mathrm{ev}(g,fx)^{-\frac{1}{2}} \cdot \rmd f_x \cdot \mathrm{ev}(g,x)^{\frac{1}{2}} ),%
\end{equation*}
where $\mathrm{ev}(g,x) := g_x$ is the evaluation map. By \cite[Rem.~2.2]{Sch}, the evaluation map is continuous, implying that $(g,x) \mapsto \mathrm{ev}(g,fx)^{-\frac{1}{2}} \cdot \rmd f_x \cdot \mathrm{ev}(g,x)^{\frac{1}{2}}$ is continuous. Since the singular values depend continuously on the matrix, then also $\beta$ is continuous.
\end{proof}

The notion of convexity to be used for $\vec{\SC}_{f,\mu}$ is a natural combination of the two well-studied notions of \emph{geodesic convexity} (see, e.g., \cite{Bou,Udr}) and \emph{cone-convexity} (see, e.g., \cite{Luc}). Recall that a function $\varphi:M \rightarrow \R$, defined on a finite-dimensional Riemannian manifold $(M,g)$, is called \emph{geodesically convex} if the composition $\varphi \circ \gamma$ with an arbitrary geodesic $\gamma$ in $M$ is convex in the usual sense. This can be reformulated as%
\begin{equation*}
  \varphi(\gamma(t)) \leq (1 - t)\varphi(\gamma(0)) + t\varphi(\gamma(1)) \mbox{\quad for all\ } t \in [0,1].%
\end{equation*}
If $\varphi$ instead takes values in a vector space equipped with a partial order $\preceq$ induced by a convex cone, and the inequality%
\begin{equation*}
  \varphi(\gamma(t)) \preceq (1 - t)\varphi(\gamma(0)) + t\varphi(\gamma(1)) \mbox{\quad for all\ } t \in [0,1]%
\end{equation*}
holds for any geodesic in $M$, we call $\varphi$ \emph{geodesically cone-convex}. It makes sense to extend this definition to functions on infinite-dimensional Riemannian manifolds, which are well-behaved in terms of their geodesics. Since this is certainly the case for the manifold $\MC_{\omega}$ with the $L^2$-metric, as Proposition \ref{prop_mcnu} shows, we can ask whether $\vec{\SC}_{f,\mu}$ is geodesically cone-convex. The following proposition shows that this is already the case for the integrand in the definition of $\vec{\SC}_{f,\mu}$.%

\begin{proposition}
For each $x \in M$, the function $g \mapsto \vec{\sigma}^g(\rmd f_x)$, $\MC_{\omega} \rightarrow \fa^+$, is geodesically cone-convex.
\end{proposition}

\begin{proof}
Let $g:[0,1] \rightarrow \MC_{\omega}$, $t \mapsto g^t$, be a geodesic segment. By Corollary \ref{cor_geodesics}, we have $g^t_x \equiv g^0_x \#_t\, g^1_x$. We prove the mid-point convexity of $g \mapsto \vec{\sigma}^g(\rmd f_x)$, i.e.%
\begin{equation}\label{eq_mp_conv_pointwise}
  \vec{\sigma}^{g^{\frac{1}{2}}}(\rmd f_x) \preceq \frac{1}{2}\vec{\sigma}^{g^0}(\rmd f_x) + \frac{1}{2}\vec{\sigma}^{g^1}(\rmd f_x) \mbox{\quad for all\ } x \in M.%
\end{equation}
To prove this, we choose charts $(\phi,U)$ and $(\psi,V)$ around $x$ and $fx$, respectively. Let $A_x$ be a local representation of $\rmd f_x$ with respect to this pair of charts. Further, let $p_0$ and $p_1$ ($q_0$ and $q_1$) represent $g^0_x$ and $g^1_x$ ($g^0_{fx}$ and $g^1_{fx}$), respectively. According to Lemma \ref{lem_singvals}, then \eqref{eq_mp_conv_pointwise} is equivalent to\footnote{Note that we are not using anymore the inversion of $p$ to define the local representation as in \eqref{eq_metric_matrix_rep}.}%
\begin{equation*}
  \vec{\sigma}([q_0 \#_{\frac{1}{2}}\, q_1]^{\frac{1}{2}}A_x [p_0 \#_{\frac{1}{2}}\, p_1]^{-\frac{1}{2}}) \preceq \frac{1}{2}\vec{\sigma}(q_0^{\frac{1}{2}} A_x p_0^{-\frac{1}{2}}) + \frac{1}{2}\vec{\sigma}(q_1^{\frac{1}{2}}A_xp_1^{-\frac{1}{2}}).%
\end{equation*}
A detailed proof of this inequality can be found in \cite[Lem.~3.2]{KHG}. The continuity, proven in Lemma \ref{lem_continuity}, together with the mid-point convexity implies convexity by a standard argument in convex analysis, see also \cite[Rem.~3.3]{KHG}.
\end{proof}

The main result of this section is the following theorem.%

\begin{theorem}\label{thm_convexity}
For every $\omega \in \VC$, the function $\vec{\SC}_{f,\mu}:\MC_{\omega} \rightarrow \fa^+$ is continuous and geodesically cone-convex.
\end{theorem}

\begin{proof}
The continuity of $\vec{\SC}_{f,\mu}$ is a simple consequence of the continuity of $(g,x) \mapsto \vec{\sigma}^g(\rmd f_x)$ together with the compactness of $M$. The cone-convexity directly follows from the cone-convexity of the integrand together with the fact that the order $\preceq$ is induced by a closed cone.
\end{proof}

Observe that the geodesic cone-convexity of $\vec{\SC}_{f,\mu}$ is equivalent to the geodesic convexity of the functions%
\begin{equation*}
  g \mapsto s_{k,f,\mu}(g) := \int \log\omega_k^g(\rmd f_x)\, \rmd\mu(x),\quad k = 1,2,\ldots,d,%
\end{equation*}
where%
\begin{equation*}
  \omega_k^g(\rmd f_x) := \prod_{i=1}^k \log\alpha_i^g(\rmd f_x).%
\end{equation*}
Moreover, a minimizer of the vector-valued optimization problem is the same as a common minimizer of the scalar optimization problems%
\begin{equation*}
  \min_{g \in \MC_{\omega}}s_{k,f,\mu}(g),\quad k = 1,2,\ldots.d.%
\end{equation*}
In particular, we obtain the following corollary.%

\begin{corollary}
Let $f$ be a $C^{1+\alpha}$-diffeomorphism and assume that $\mu$ is an ergodic SRB measure. Let $k$ denote the number of positive Lyapunov exponents of $(M,f,\mu)$. Then, the measure-theoretic entropy $h_{\mu}(f)$ can be written as the solution to a geodesically convex optimization problem as follows:%
\begin{equation*}
  h_{\mu}(f) = \inf_{g \in \MC_{\omega}}s_{k,f,\mu}(g).%
\end{equation*}
\end{corollary}

\begin{proof}
If $\mu$ is an ergodic measure, the measure-theoretic entropy of $f$ with respect to $\mu$ is given by%
\begin{equation*}
  h_{\mu}(f) = \sum_{i=1}^d \max\{0,\lambda_i(f)\} = \sum_{i=1}^k \lambda_i(f)%
\end{equation*}
according to the well-known entropy-characterization of SRB measures by Ledrappier and Young \cite{LYo} (an extension of Pesin's entropy formula). Together with Theorem \ref{thm_bochi}, this immediately implies the statement of the corollary.
\end{proof}

\begin{remark}
The corollary can probably be extended to the case when $f$ is only a local $C^{1+\alpha}$-diffeomorphism; see \cite{Mea} for a corresponding entropy formula, which does not assume invertibility of $f$.
\end{remark}

Convex optimization problems are more well-behaved than general (nonlinear) optimization problems in many respects. In particular, every local optimum is also a global optimum, the solution set is convex, there exist efficient numerical algorithms to solve them, and in the case of differentiability there exist first-order optimality criteria. However, the convex optimization problem considered in this paper still has a number of disadvantages which makes it hard to solve:%
\begin{itemize}
\item The existence of a minimizer is not guaranteed.%
\item It is vector-valued.%
\item It is defined on an infinite-dimensional space.%
\item It is defined on a Riemannian manifold (rather than a vector space).%
\item It cannot be expected to be smooth, since singular value functions are not smooth, in general.%
\end{itemize}
Despite these drawbacks, the optimization problem also has some properties which could turn out to be advantageous. For instance, we can prove that the functions $s_{k,f,\mu}$ satisfy global Lipschitz estimates on $\MC_{\omega}$ with respect to the $L^2$-metric.%

\begin{theorem}\label{thm_lipschitz_l2}
For each $k \in \{1,2,\ldots,d\}$, the function $s_{k,f,\mu}:\MC_{\omega} \rightarrow \R$ satisfies the global Lipschitz-like estimate%
\begin{equation}\label{eq_lip_ineq1}
  |s_{k,f,\mu}(g^1) - s_{k,f,\mu}(g^2)| \leq \sqrt{k}\left(\int d(g^1_x,g^2_x)^2\, \rmd\mu(x)\right)^{\frac{1}{2}},%
\end{equation}
where $d(\cdot,\cdot)$ is the distance function on $\SC^+_d$. If $\mu$ is a volume measure induced by some Riemannian metric $g^0 \in \MC_{\omega}$, then%
\begin{equation}\label{eq_lip_ineq2}
  |s_{k,f,\mu}(g^1) - s_{k,f,\mu}(g^2)| \leq \sqrt{k}\, d_{L^2}(g^1,g^2).%
\end{equation}
That is, $s_{k,f,\mu}$ satisfies a global Lipschitz estimate with respect to the $L^2$-distance.
\end{theorem}

\begin{proof}
First, observe that by Jensen's inequality we have%
\begin{equation*}
  (s_{k,f,\mu}(g^1) - s_{k,f,\mu}(g^2))^2 \leq \int_M \log^2 \frac{\omega_k^{g^1}(\rmd f_x)}{\omega_k^{g^2}(\rmd f_x)}\, \rmd\mu(x).%
\end{equation*}
We split $M$ into the two subsets%
\begin{equation*}
  M_1 := \left\{ x \in M : \omega_k^{g^1}(\rmd f_x) > \omega_k^{g^2}(\rmd f_x) \right\},\quad M_2 := M \backslash M_1.%
\end{equation*}
Then we have the following identities and inequalities:%
\begin{align*}
& \int_{M_1} \log^2 \frac{\omega_k^{g^1}(\rmd f_x)}{\omega_k^{g^2}(\rmd f_x)}\, \rmd\mu(x) \\
&= \int_{M_1} \log^2 \frac{\omega_k( (g^1_{fx})^{-\frac{1}{2}} \rmd f_x (g^1_x)^{\frac{1}{2}} )}{\omega_k( (g^2_{fx})^{-\frac{1}{2}} \rmd f_x (g^2_x)^{\frac{1}{2}} )}\, \rmd\mu(x) \\
					&= \int \log^2 \frac{\omega_k( (g^1_{fx})^{-\frac{1}{2}} (g^2_{fx})^{\frac{1}{2}} (g^2_{fx})^{-\frac{1}{2}} \rmd f_x (g^2_x)^{\frac{1}{2}} (g^2_x)^{-\frac{1}{2}} (g^1_x)^{\frac{1}{2}})}{\omega_k( (g^2_{fx})^{-\frac{1}{2}}\rmd f_x (g^2_x)^{\frac{1}{2}} )}\, \rmd\mu(x) \\
&\leq \int_{M_1} \left( \log \omega_k( (g^1_{fx})^{-\frac{1}{2}} (g^2_{fx})^{\frac{1}{2}} ) + \log\omega_k( (g^2_x)^{-\frac{1}{2}} (g^1_x)^{\frac{1}{2}}) \right)^2 \rmd\mu(x) \\
&\leq 2\int_{M_1}\left(\log^2 \omega_k( (g^1_{fx})^{-\frac{1}{2}} (g^2_{fx})^{\frac{1}{2}} ) + \log^2\omega_k( (g^2_x)^{-\frac{1}{2}} (g^1_x)^{\frac{1}{2}})\right)\, \rmd \mu(x).%
\end{align*}
In the first equality, we simply use the definition of $s_{k,f,\mu}$. The second equality identifies $\rmd f_x$, $g^1_x$ and $g^2_x$ with their corresponding local representations and uses Lemma \ref{lem_singvals}. The third equality is trival. The subsequent inequality is Horn's inequality together with the fact that $\log^2(\cdot)$ is increasing on $(1,\infty)$. The last inequality simply uses that $(x + y)^2 \leq 2x^2 + 2y^2$ for any $x,y \in \R$.%

Now, observe that for any $a,b \in \SC^+_d$ we have%
\begin{align*}
  \log^2 \omega_k( a^{-\frac{1}{2}}b^{\frac{1}{2}} ) &= \left(\sum_{i=1}^k \log\alpha_i( a^{-\frac{1}{2}}b^{\frac{1}{2}} )\right)^2 \\
	&\leq k \sum_{i=1}^k \log^2\alpha_i( a^{-\frac{1}{2}}b^{\frac{1}{2}} ) \leq k \sum_{i=1}^d \log^2\alpha_i( a^{-\frac{1}{2}}b^{\frac{1}{2}} ) \\
	&= k \|\vec{\sigma}(a^{-\frac{1}{2}}b^{\frac{1}{2}} )\|_2^2 = \frac{k}{4} \|\vec{d}(a,b)\|_2^2 = \frac{k}{4} d(a,b)^2.%
\end{align*}
Altogether, we obtain the estimate%
\begin{equation*}
  \int_{M_1} \log^2 \frac{\omega_k^{g^1}(\rmd f_x)}{\omega_k^{g^2}(\rmd f_x)}\, \rmd\mu(x) \leq \frac{k}{2}\int_{M_1} \left(d(g^1_{fx},g^2_{fx})^2 + d(g^1_x,g^2_x)^2\right)\, \rmd\mu(x).%
\end{equation*}
In a similar fashion, we derive that%
\begin{equation*}
  \int_{M_2} \log^2 \frac{\omega_k^{g^1}(\rmd f_x)}{\omega_k^{g^2}(\rmd f_x)}\, \rmd\mu(x) \leq \frac{k}{2}\int_{M_2} \left(d(g^2_{fx},g^1_{fx})^2 + d(g^2_x,g^1_x)^2\right)\, \rmd\mu(x).%
\end{equation*}
Because of the symmetry of the distance and the $f$-invariance of $\mu$, summing the two integrals gives%
\begin{align*}
  (s_{k,f,\mu}(g^1) - s_{k,f,\mu}(g^2))^2 \leq k \int_M d(g^1_x,g^2_x)^2\, \rmd\mu(x),%
\end{align*}
which is equivalent to the claimed inequality \eqref{eq_lip_ineq1}.%

To prove \eqref{eq_lip_ineq2}, assume that $\mu$ is the volume measure induced by some $g^0 \in \MC_{\omega}$. According to \eqref{eq_lip_ineq1} and \eqref{eq_l2_distance}, it suffices to show that $d_x(g^1_x,g^2_x) = d(g^1_x,g^2_x)$ for all $g^1,g^2 \in \MC_{\omega}$ and $x \in M$. By definition, $d_x(\cdot,\cdot)$ is the distance function on $(S^0_2M)_x$ induced by the Riemannian metric%
\begin{equation*}
  \langle a,b \rangle_p = \tr(p^{-1}ap^{-1}b)\sqrt{\det((g^0_x)^{-1}p)},%
\end{equation*}
while $d(\cdot,\cdot)$ is the distance function induced by%
\begin{equation*}
  \langle a,b \rangle_p' = \tr(p^{-1}ap^{-1}b).%
\end{equation*}
We identify the space of inner products on $T_xM$ with $\SC_d^+$. The submanifold of $\SC_d^+$ consisting of all matrices with the same determinant $c > 0$ is totally geodesic. Hence, if $\gamma:[0,1] \rightarrow \SC^+_d$ is a geodesic segment connecting $g^1_x$ with $g^2_x$, then $\det(\gamma(t))$ is the same for all $t \in [0,1]$ and equals $\det(g^0_x)$. It follows that $\det((g^0_x)^{-1}\gamma(t)) = 1$ for all $t$, showing that $\langle \cdot,\cdot \rangle' = \langle \cdot,\cdot \rangle$, and hence $d_x(\cdot,\cdot) = d(\cdot,\cdot)$.
\end{proof}

\begin{remark}
The proven fact that $s_{k,f,\mu}$ is globally Lipschitz on $\MC_{\omega}$ if $\mu$ is a volume measure implies that there exists a unique extension of $s_{k,f,\mu}$ to the metric completion of $\MC_{\omega}$ which obeys the same Lipschitz estimate. The completion of $\MC$ was studied in \cite{Cl3} and, roughly speaking, characterized as the space of measurable metrics inducing a finite volume. In fact, it is isometric to the space of $L^2$-mappings from the manifold $M$ to the completion of the space of positive-definite symmetric matrices, as proven in \cite{Cav}.
\end{remark}

\section{Towards a first-order optimality criterion}\label{sec_first_order}

Since singular value functions do not depend smoothly on parameters, we cannot expect the function $\vec{\SC}_{f,\mu}$ (or $s_{k,f,\mu}$) to be differentiable. However, the geodesic convexity implies that unilateral directional derivatives exist (see \cite[Ch.~3, Thm.~4.2]{Udr}) and possibly we can compute some generalized derivative, e.g., a subdifferential. This would lead to a first-order criterion for optimality.%

In this section, we only take the first step and compute directional derivatives of the functions $s_{k,f,\mu}$. Moreover, we restrict ourselves to the simplest situation in which all singular values have multiplicity one at each point of the manifold. Although this assumption is not very realistic, it will give us some idea of what we can expect a subgradient of $s_{k,f,\mu}$ to look like.%

\begin{theorem}\label{thm_derivative}
Let $g \in \MC_{\omega}$ be a metric satisfying $\alpha_k^g(\rmd f_x) > \alpha_{k+1}^g(\rmd f_x)$ for all $x \in M$ and some $k \in \{1,2,\ldots,d-1\}$. Then, all directional derivatives of $s_{k,f,\mu}$ at $g$ exist. For any $h \in T_g\MC_{\omega}$, the directional derivative of $s_{k,f,\mu}$ at $g$ in direction $h$ is given by%
\begin{equation*}
  \partial_h s_{k,f,\mu}(g) = \int \langle Q^{\mathrm{l}}_{k,x} - Q^{\mathrm{r}}_{k,x}, h_x \rangle_{g_x}\, \rmd\mu(x),%
\end{equation*}
where $Q^{\mathrm{l}}_{k,x}$ ($Q^{\mathrm{r}}_{k,x}$) denotes the orthogonal projection in $(T_xM,g_x)$ onto the subspace spanned by the first $k$ left (right) singular vectors of $\rmd f_{f^{-1}x}$ ($\rmd f_x$).%
\end{theorem}

\begin{proof}
The proof is subdivided into six steps.%

\emph{Step 1}: Let us write $s_{k,x}(g) := \log\omega_k^g(\rmd f_x)$, which is the integrand in the definition of $s_{k,f,\mu}(g)$. We first try to differentiate $s_{k,x}$ in the hope that the order of differentiation and integration can be interchanged.  To this end, we choose a smooth curve $\gamma$ in $\MC_{\omega}$ with $\gamma(0) = g$, $\dot{\gamma}(0) = h$ and decompose the function $s_{k,x} \circ \gamma$ as%
\begin{equation*}
  (s_{k,x} \circ \gamma)(t) = \ell_k \circ \alpha \circ \zeta_x(t),%
\end{equation*}
where%
\begin{align*}
  \zeta_x(t) &= \gamma(t)_{fx}^{\frac{1}{2}} \rmd f_x \gamma(t)_x^{-\frac{1}{2}}, \quad \zeta_x:\R \rightarrow \GL(d,\R), \\
  \alpha(a) &:= (\alpha_1(a),\ldots,\alpha_d(a)), \quad \alpha:\GL(d,\R) \rightarrow \R^d, \\
  \ell_k(\xi) &:= \sum_{i=1}^k \log\hat{\xi}_i, \quad \ell_k:\R^d \rightarrow [-\infty,\infty).%
\end{align*}
Here, some explanation is necessary. In the definition of $\zeta_x(t)$, we treat $\gamma(t)_x$, $\gamma(t)_{fx}$ and $\rmd f_x$ as matrices, where we think of their local representations with respect to appropriate charts (with Lemma \ref{lem_singvals} in mind). The function $\alpha$ simply maps a matrix $a \in \GL(d,\R)$ to the vector of its ordinary singular values. Finally, for a vector $\xi = (\xi_1,\ldots,\xi_d) \in \R^d$, we write $\hat{\xi}_1 \geq \cdots \geq \hat{\xi}_d$ for the absolute values of its components, ordered from the largest to the smallest, and then define $\ell_k(\xi)$ as above. Here, the value $-\infty$ is possible if at least one component of $\xi$ is zero. Using Lemma \ref{lem_singvals}, we then see that $s_{k,x}$ can be decomposed in the suggested way (keeping in mind that we are not using anymore the inversion of $p$ to define the local representation as in \eqref{eq_metric_matrix_rep}).%

\emph{Step 2}: We compute the derivative of $\zeta_x$ at $t = 0$. By the product rule, we have%
\begin{equation*}
  \dot{\zeta}_x(t) = \left[\frac{\partial}{\partial t}\Bigl|_{t=0} \gamma(t)_{fx}^{\frac{1}{2}}\right] \rmd f_x \gamma(0)_x^{-\frac{1}{2}} + \gamma(0)_{fx}^{\frac{1}{2}} \rmd f_x \left[\frac{\partial}{\partial t}\Bigl|_{t=0} \gamma(t)_x^{-\frac{1}{2}}\right].%
\end{equation*}
With the notation ${\bf r}:p \mapsto p^{\frac{1}{2}}$ and ${\bf i}:p \mapsto p^{-1}$ (both defined on $\SC^+_d$), we can write this as%
\begin{equation*}
  \dot{\zeta}_x(t) = \left[\rmd{\bf r}_{\gamma(0)_{fx}}\dot{\gamma}(0)_{fx}\right] \rmd f_x \gamma(0)_x^{-\frac{1}{2}} + \gamma(0)_{fx}^{\frac{1}{2}} \rmd f_x \left[\rmd({\bf i} \circ {\bf r})_{\gamma(0)_x}\dot{\gamma}(0)_x\right].%
\end{equation*}
Using that $\rmd{\bf i}_ph = -p^{-1}hp^{-1}$ and $\rmd{\bf r}_ph = L_p^{-1}(h)$, where $L_p$ is the Lyapunov operator\footnote{Since $p^{\frac{1}{2}}$ has only strictly positive eigenvalues, the Lyapunov operator is invertible.}%
\begin{equation*}
  L_p(X) := p^{\frac{1}{2}}X + Xp^{\frac{1}{2}},%
\end{equation*}
we can finally write the derivative of $\zeta_x$ as%
\begin{equation}\label{eq_zeta_derivative}
  \dot{\zeta}_x(t) = L_{g_{fx}}^{-1}(h_{fx}) \rmd f_x g_x^{-\frac{1}{2}} - g_{fx}^{\frac{1}{2}} \rmd f_x g_x^{-\frac{1}{2}} L_{g_x}^{-1}(h_x)g_x^{-\frac{1}{2}}.%
\end{equation}
Observe that $X := L_{g_x}^{-1}(h_x)$ is a symmetric matrix, since it solves $g_x^{\frac{1}{2}}X + Xg_x^{\frac{1}{2}} = h_x$ and by symmetry of $p$ and $h_x$, also $X\trn$ solves this equation. Hence, $X = L_{g_x}^{-1}(h_x) = X\trn$.%

\emph{Step 3}: Now, we consider the composed function $\ell_k \circ \alpha$. According to \cite[Prop.~6.2]{LSe}, $\ell_k \circ \alpha$ is differentiable at $X \in \R^{d\tm d}$ if and only if $\ell_k$ is differentiable at $\alpha(X)$. The latter is the case if $\alpha_k(X) > \alpha_{k+1}(X)$, which holds in our case by assumption, since $\alpha_k(X) = \alpha_k(\zeta_x(0)) = \alpha_k^g(\rmd f_x)$. In this case,%
\begin{equation*}
  \nabla (\ell_k \circ \alpha)(X) = U\trn\Diag\left(\frac{1}{\alpha_1(X)},\ldots,\frac{1}{\alpha_k(X)},0,\ldots,0\right)V,%
\end{equation*}
where $X = U\trn\Diag(\alpha(X))V$ is a singular value decomposition.%

\emph{Step 4}: To compute the derivative of $s_{k,x} \circ \gamma$, we need to compose the two derivatives $\nabla(\ell_k \circ \alpha)(\zeta_x(0))$ and $\dot{\zeta}_x(0)$ in the proper way. This is the Euclidean inner product on $\R^{d \tm d}$, which is given by $\langle X,Y \rangle = \tr[X\trn Y]$ (also known as \emph{Frobenius inner product}). For brevity, let us write $S_x := \nabla(\ell_k \circ \alpha)(\zeta_x(0))$ and $Z_x := L_{g_x}^{-1}(h_x)$. Then%
\begin{align*}
 & \frac{\partial}{\partial t}\Bigl|_{t=0}(s_{k,x} \circ \gamma)(0) = \tr[S_x\trn\dot{\zeta}_x(0)] \\
	                       &\stackrel{\eqref{eq_zeta_derivative}}{=} \tr\left[S_x\trn \left(Z_{fx}\rmd f_x g_x^{-\frac{1}{2}} - g_{fx}^{\frac{1}{2}} \rmd f_x g_x^{-\frac{1}{2}} Z_x g_x^{-\frac{1}{2}}\right)\right] \\
												 &= \tr\left[S_x\trn Z_{fx}\rmd f_x g_x^{-\frac{1}{2}} \right] - \tr\left[S_x\trn g_{fx}^{\frac{1}{2}} \rmd f_x g_x^{-\frac{1}{2}} Z_x g_x^{-\frac{1}{2}} \right] \\
    										 &= \tr\left[Z_{fx}\rmd f_x g_x^{-\frac{1}{2}}S_x\trn\right] - \tr\left[Z_x g_x^{-\frac{1}{2}}S_x\trn g_{fx}^{\frac{1}{2}} \rmd f_x g_x^{-\frac{1}{2}}\right] \\
                         &= \tr\left[Z_{fx}g_{fx}^{-\frac{1}{2}}\zeta_x(0) S_x\trn\right] - \tr\left[Z_x g_x^{-\frac{1}{2}} S_x\trn \zeta_x(0)\right].%
\end{align*}
From the definition of $S_x$, it follows that%
\begin{align*}
  \zeta_x(0) S_x\trn = U_{\zeta_x(0)}\trn (I_{k \tm k} \oplus 0_{(d-k)\tm(d-k)}) U_{\zeta_x(0)},%
\end{align*}
where $\zeta_x(0) = U_{\zeta_x(0)}\trn \Diag(\alpha(\zeta_x(0))) V_{\zeta_x(0)}$ is the chosen singular value decomposition. Analogously,%
\begin{equation*}
  S_x\trn \zeta_x(0) = V_{\zeta_x(0)}\trn (I_{k \tm k} \oplus 0_{(d-k)\tm(d-k)}) V_{\zeta_x(0)}.%
\end{equation*}
We can write these matrices in a more compact way as $\bar{U}_{fx}\bar{U}_{fx}\trn$ and $\bar{V}_x\bar{V}_x\trn$, where $\bar{U}_{fx}$ is the $d \tm k$ matrix whose columns are the first $k$ left singular vectors of $\zeta_x(0)$, and $\bar{V}_x$ is defined in the same way via the right singular vectors. Observe that $\bar{U}_{fx}\bar{U}_{fx}\trn$ is the orthogonal projection onto the subspace spanned by the first $k$ left singular vectors of $\zeta_x(0)$ (analogously for $\bar{V}_x\bar{V}_x\trn$). We thus have%
\begin{equation*}
  \frac{\partial}{\partial t}\Bigl|_{t=0}(s_{k,x} \circ \gamma)(0) = \tr\left[Z_{fx}g_{fx}^{-\frac{1}{2}}\bar{U}_{fx}\bar{U}_{fx}\trn\right] - \tr\left[Z_x g_x^{-\frac{1}{2}}\bar{V}_x\bar{V}_x\trn\right].%
\end{equation*}
Now, recall that $g_x^{\frac{1}{2}}Z_x + Z_xg_x^{\frac{1}{2}} = h_x$, implying $Z_xg_x^{-\frac{1}{2}} + g_x^{-\frac{1}{2}}Z_x = g_x^{-\frac{1}{2}}h_x g_x^{-\frac{1}{2}}$. Observe that for arbitrary symmetric matrices $a,p,x,y$ with $px + xp = y$, we have $\tr[apx] = \tr[(apx)\trn] = \tr[xpa] = \tr[axp]$, and hence $\tr[ay] = \tr[a(px)] + \tr[a(xp)] = 2\tr[apx]$. This yields%
\begin{align}\label{eq_der_final}
\begin{split}
  &\frac{\partial}{\partial t}\Bigl|_{t=0}(s_{k,x} \circ \gamma)(0) = \frac{1}{2}\tr\left[g_{fx}^{-\frac{1}{2}} h_{fx} g_{fx}^{-\frac{1}{2}}\bar{U}_{fx}\bar{U}_{fx}\trn \right] - \frac{1}{2}\tr\left[g_x^{-\frac{1}{2}} h_x g_x^{-\frac{1}{2}}\bar{V}_x\bar{V}_x\trn\right] \\
	&= \frac{1}{2}\tr\left[g_{fx}^{\frac{1}{2}} g_{fx}^{-1} h_{fx} g_{fx}^{-1} g_{fx}^{\frac{1}{2}}\bar{U}_{fx}\bar{U}_{fx}\trn\right] - \frac{1}{2}\tr\left[g_x^{\frac{1}{2}} g_x^{-1} h_x g_x^{-1}g_x^{\frac{1}{2}}\bar{V}_x\bar{V}_x\trn\right] \\
	&= \frac{1}{2} \langle g_{fx}^{\frac{1}{2}}\bar{U}_{fx}\bar{U}_{fx}\trn g_{fx}^{\frac{1}{2}}, h_{fx} \rangle_{g_{fx}} - \frac{1}{2} \langle g_x^{\frac{1}{2}} \bar{V}_x\bar{V}_x\trn g_x^{\frac{1}{2}}, h_x \rangle_{g_x}.
\end{split}
\end{align}

\emph{Step 5}: We have to understand how the matrices $\bar{U}_{fx}\bar{U}_{fx}\trn$ and $\bar{V}_x\bar{V}_x\trn$ change, when we change the charts. Recall that $\bar{U}_{fx}$ is a matrix built from left singular vectors of $\zeta_x(0)$, i.e.~the eigenvectors of $\zeta_x(0)\zeta_x(0)\trn$. One can easily show that with respect to another pair of charts, the corresponding symmetric matrix becomes%
\begin{equation*}
  \tilde{\zeta}_x(0)\tilde{\zeta}_x(0)\trn = o \zeta_x(0)\zeta_x(0)\trn o\trn%
\end{equation*}
with an orthogonal matrix $o$ of the form%
\begin{equation*}
  o = (b \ast g_{fx})^{-\frac{1}{2}} b g_{fx}^{\frac{1}{2}},%
\end{equation*}
where $b$ is the derivative of the coordinate change at $fx$. Then, in the new pair of charts, $\bar{U}_{fx}\bar{U}_{fx}\trn$ becomes $o \bar{U}_{fx}\bar{U}_{fx}\trn o\trn$. Hence,%
\begin{equation*}
  (b \ast g_{fx})^{\frac{1}{2}} o\bar{U}_{fx}\bar{U}_{fx}\trn o\trn (b \ast g_{fx})^{\frac{1}{2}} = b g_{fx}^{\frac{1}{2}} \bar{U}_{fx}\bar{U}_{fx}\trn g_{fx}^{\frac{1}{2}} b\trn = b \ast g_{fx}^{\frac{1}{2}} \bar{U}_{fx}\bar{U}_{fx}\trn g_{fx}^{\frac{1}{2}}.%
\end{equation*}
This computation shows that $g_{fx}^{\frac{1}{2}} \bar{U}_{fx}\bar{U}_{fx}\trn g_{fx}^{\frac{1}{2}}$ is the local representation of a global section of $S^0_2M$, i.e.~a symmetric (not necessarily smooth) $(0,2)$-tensor field. The same can be shown for $g_x^{\frac{1}{2}}\bar{V}_x\bar{V}_x\trn g_x^{\frac{1}{2}}$. What are these global sections? Recall from the proof of Lemma \ref{lem_singvals} that a local representation of $\rmd f_x^* \rmd f_x$ is given by $p_x^{-1}A_x\trn p_{fx} A_x$, where $A_x$ is a local representation of $\rmd f_x$ and $p_x$ ($p_{fx}$) one of $g_x$ ($g_{fx}$). The eigenvectors $v_1,\ldots,v_d$ of $\rmd f_x^* \rmd f_x$ then have local representations $\tilde{v}_1,\ldots,\tilde{v}_d$, which are the eigenvectors of $p_x^{-1}A_x\trn p_{fx} A_x =: C_x$. The matrix $\zeta_x(0)\trn \zeta_x(0)$ is then given by $p_x^{\frac{1}{2}}C_xp_x^{-\frac{1}{2}}$, and hence its eigenvectors are $p_x^{\frac{1}{2}}\tilde{v}_1,\ldots,p_x^{\frac{1}{2}}\tilde{v}_d$ (which form an orthogonal basis with respect to the standard Euclidean inner product). If $Q_x$ is the orthogonal projection onto the linear subspace of $\R^d$ spanned by the first $k$ eigenvectors of $\zeta_x(0)\trn \zeta_x(0)$, then for $i = 1,\ldots,k$ we have $(p_x^{\frac{1}{2}} Q_x p_x^{\frac{1}{2}}) \tilde{v}_i = p_x \tilde{v}_i$ and for $i = k+1,\ldots,d$ we have $(p_x^{\frac{1}{2}} Q_x p_x^{\frac{1}{2}}) \tilde{v}_i = 0$. Hence,%
\begin{align*}
  \langle (p_x^{\frac{1}{2}} Q_x p_x^{\frac{1}{2}}) \tilde{v}_i, \tilde{v}_j \rangle = \left\{ \begin{array}{cc} \langle (p_x^{\frac{1}{2}} \tilde{v}_i), (p_x^{\frac{1}{2}}\tilde{v}_j) \rangle = \delta_{ij} & \mbox{for } i = 1,\ldots,k, \\
	0 & \mbox{otherwise}.
	\end{array}\right.%
\end{align*}
It follows that $g_{fx}^{\frac{1}{2}}\bar{U}_{fx}\bar{U}_{fx}\trn g_{fx}^{\frac{1}{2}}$ is the local representation of the orthogonal projection $Q^{\mathrm{l}}_{k,fx}$ onto the subspace of $T_{fx}M$ which is spanned by the first $k$ left singular vectors of $\rmd f_x$. Analogously, $g_x^{\frac{1}{2}}\bar{V}_x\bar{V}_x\trn g_x^{\frac{1}{2}}$ represents the orthogonal projection $Q^{\mathrm{r}}_{k,x}$ onto the subspace of $T_xM$ spanned by the first $k$ right singular vectors of $\rmd f_x$. Here, of course, the orthogonality is defined in terms of the metric $g = \gamma(0)$.%

\emph{Step 6}: Using the invariance of the measure $\mu$, we obtain from \eqref{eq_der_final} via integration that%
\begin{align*}
  \int \frac{\partial}{\partial t}\Bigl|_{t=0}(s_{k,x} \circ \gamma)(0)\, \rmd\mu(x) = \frac{1}{2} \int \langle Q^{\mathrm{l}}_{k,x} - Q^{\mathrm{r}}_{k,x}, h_x \rangle_{g_x}\, \rmd\mu(x).%
\end{align*}
It remains to prove that integration and derivative on the left-hand side can be interchanged. To show this, first observe that our assumption about the spectral gap carries over from $g$ to nearby metrics $\gamma(t)$, $|t| \leq \ep$, by Lemma \ref{lem_continuity}. Hence, the above formula holds for all $(t,x) \in [-\ep,\ep] \tm M$. Thus, the derivative of $s_{k,x} \circ \gamma$ at any $t \in [-\ep,\ep]$ has the same form as above, but the orthogonal projections are then defined in terms of $\gamma(t)$ instead of $\gamma(0)$:%
\begin{equation*}
  \frac{\partial}{\partial t}\Bigl|_{t=0}(s_{k,x} \circ \gamma)(t) = \frac{1}{2} \langle Q^{\mathrm{l}}_{k,x}(t) - Q^{\mathrm{r}}_{k,x}(t), \dot{\gamma}(t)_x \rangle_{\gamma(t)_x}.%
\end{equation*}
Here, we can estimate%
\begin{align*}
  &\left|\langle Q^{\mathrm{l}}_{k,x}(t) - Q^{\mathrm{r}}_{k,x}(t), \dot{\gamma}(t)_x \rangle_{\gamma(t)_x}\right| \\
		&\leq \left(\| Q^{\mathrm{l}}_{k,x}(t) \|_{\gamma(t)_x} + \| Q^{\mathrm{r}}_{k,x}(t) \|_{\gamma(t)_x}\right) \cdot \|\dot{\gamma}(t)_x\|_{\gamma(t)_x}.%
\end{align*}
The term $\|\dot{\gamma}(t)_x\|_{\gamma(t)_x}$ is continuous in $(t,x)$, and thus can be estimated by a constant $C$ on a compact set of the form $[-\ep,\ep] \tm M$. The orthogonal projections have norm $1$, hence we can estimate the complete term by a constant. We conclude that the Leibniz rule can be applied to interchange the order of integration and differentiation. The proof is complete.
\end{proof}

\begin{remark}
In the case when $\mu$ is a volume measure, coming from a metric which induces the volume form $\omega$, we would expect (in the setup of the theorem) that $s_{k,f,\mu}$ is differentiable at $g$ with gradient $x \mapsto Q^{\mathrm{l}}_{k,x} - Q^{\mathrm{r}}_{k,x}$ (which may be nonsmooth and live in some metric space completion). Hence, a vanishing gradient would mean that $Q^{\mathrm{l}}_{k,x} = Q^{\mathrm{r}}_{k,x}$, which is some sort of alignment of derivatives. If this holds for all $k$, it can also be interpreted as a symmetry property, since the left and right singular vectors of a real square matrix are the same if and only if the matrix is symmetric. This is only an analogy, however, because here we are not dealing with a single matrix but the composition of different linear operators (except at the fixed points of $f$). 
\end{remark}

\section{Summary and outlook}\label{sec_final}

We have shown that the vector of averaged Lyapunov exponents of a smooth measure-preserving system on a compact manifold $M$ is the infimum of a geodesically cone-convex function, defined on the space of smooth Riemannian metrics on $M$ which preserve a given volume form. Several research directions are conceivable to make this a fruitful approach for a better theoretical understanding or numerical computations. These include:%
\begin{enumerate}
\item[1.] \emph{Finding criteria for the existence of minimizers}; if such minimizers can be shown to exist only in the completion of $\MC_{\omega}$, an important question concerns their regularity.%
\item[2.] \emph{Computation of subgradients for the derivation of a first-order optimality criterion}. Maybe, in some cases it is possible to find an optimal metric by solving the equation defined by the first-order criterion.%
\item[3.] \emph{Finding ways to restrict the optimization problem to one or several finite-dimensional convex problems}. The approach for doing this that comes first to mind is probably restricting the function $\vec{\SC}_{f,\mu}$ to finite-dimensional totally geodesic submanifolds of $\MC_{\omega}$. However, it is very likely that there are none except for possibly low-dimensional ones. One way to find such submanifolds is to consider fixed point components of isometries. If we consider $\MC$ instead of $\MC_{\omega}$, we know that the diffeomorphism group of $M$ acts on $\MC$ by isometries. A fixed point of the action of one diffeomorphism $\varphi$ is a metric in which $\varphi$ is an isometry. So, one possible construction of a totally geodesic submanifold is%
\begin{equation*}
  S := \{ g \in \MC : f \in \mathrm{Isom}(M,g) \mbox{\ for all\ } f \in F \},%
\end{equation*}
where $F \subset \mathrm{Diff}(M)$ is any subset. As an example, consider a manifold $M$ which is a Lie group and let $F$ be the group of left translations. Then $S$ is the space of left-invariant metrics which is finite-dimensional and can be identified with $\SC^+_d$. In general, to obtain a finite-dimensional $S$ by the above construction, one would expect that the orbits $Fx$ all have to be dense in $M$. But then, any metric in $S$ is determined by its value at a single point. Another way to obtain a finite-dimensional problem is to start with low-dimensional totally geodesic submanifold, solve the problem there, then use the solution as the initial guess for another finite-dimensional problem (in some other submanifold) and continue like this in the hope that the process converges to a globally (nearly) optimal metric.
\end{enumerate}


\begin{thebibliography}{99}
\bibitem{Ani} M.~Anikushin. \emph{Variational description of uniform Lyapunov exponents via adapted metrics on exterior products}. Nonlinear Differential Equations and Applications NoDEA 33(1), 26, 2026.%
\bibitem{ARo} M.~Anikushin, A.~Romanov. \emph{Robust upper estimates for topological entropy via nonlinear constrained optimization over adapted metrics}. arXiv preprint arXiv:2503.11150, 2025.%
\bibitem{Bha} R.~Bhatia. \emph{Positive Definite Matrices}. Princeton University Press, 2009.%
\bibitem{BXY} A.~Blumenthal, J.~Xue, L.--S.~Young. \emph{Lyapunov exponents for random perturbations of some area-preserving maps including the standard map}. Annals of Mathematics 185(1), 285--310, 2017.%
\bibitem{Boc} J.~Bochi. \emph{Ergodic optimization of Birkhoff averages and Lyapunov exponents}. Proceedings of the International Congress of Mathematicians (ICM 2018), 1825--1846, 2018.%
\bibitem{BNa} J.~Bochi, A.~Navas. \emph{A geometric path from zero Lyapunov exponents to rotation cocycles}. Ergodic Theory and Dynamical Systems, vol.~35, no.~2, 374--402, 2015.%
\bibitem{BLR} V.~A.~Boichenko, G.~A.~Leonov, V.~Reitmann. \emph{Dimension Theory for Ordinary Differential Equations}. Teubner, Stuttgart, 2005.%
\bibitem{Bou} N.~Boumal. \emph{An Introduction to Optimization on Smooth Manifolds}. Cambridge University Press, 2023.%
\bibitem{Cav} N.~Cavallucci. \emph{The $L^2$-completion of the space of Riemannian metrics is CAT (0): A shorter proof}. Proceedings of the American Mathematical Society 151(7), 3183--3187, 2023.%
\bibitem{CSu} N.~Cavallucci, Z.~Su. \emph{The metric completion of the space of vector-valued one-forms}. Annals of Global Analysis and Geometry, vol.~64, no.~2, 10, 2023.%
\bibitem{Cla} B.~Clarke. \emph{The metric geometry of the manifold of Riemannian metrics over a closed manifold}. Calculus of Variations and 
Partial Differential Equations 39, 533--545, 2010.%
\bibitem{Cl2} B.~Clarke. \emph{Geodesics, distance, and the CAT (0) property for the manifold of Riemannian metrics}. Mathematische Zeitschrift 273(1--2), 55--93, 2013.%
\bibitem{Cl3} B.~Clarke. \emph{The completion of the manifold of Riemannian metrics}. Journal of Differential Geometry 93(2), 203--268, 2013.%
\bibitem{Ebi} D.~G.~Ebin. \emph{The manifold of Riemannian metrics}. In: Global Analysis, Berkeley, Calif., 1968. Proc.~Sympos.~Pure Math., vol.~15, 1970.%
\bibitem{FGr} D.~S.~Freed, D.~Groisser. \emph{The basic geometry of the manifold of Riemannian metrics and of its quotient by the diffeomorphism group}. Michigan Mathematical Journal 36(3), 323--344, 1989.%
\bibitem{GHK} P.~Giesl, S.~Hafstein, C.~Kawan. \emph{Review on contraction analysis and computation of contraction metrics}. Journal of Computational Dynamics 10(1), 1--47, 2023.%
\bibitem{Gea} P.~Giesl, S.~Hafstein, M.~Haraldsdottir, D.~Thorsteinsson, C.~Kawan. \emph{Subgradient algorithm for computing contraction metrics for equilibria}. Journal of Computational Dynamics 10(2), 281--303, 2023.%
\bibitem{GMi} O.~Gil-Medrano, P.~W.~Michor. \emph{The Riemannian manifold of all Riemannian metrics}. Q.~J.~Math., Oxf.~II. Ser.~42, No.~166, 183--202, 1991.%
\bibitem{Gou} N.~Gourmelon. \emph{Adapted metrics for dominated splittings}. Ergodic Theory and Dynamical Systems 27(6), 1839--1849, 2007.%
\bibitem{HKa} B.~Hasselblatt, A.~Katok. \emph{Principal Structures}. Handbook of dynamical systems. Vol.~1. Elsevier Science, 2002. 1--203.%
\bibitem{KHa} A.~Katok, B.~Hasselblatt. \emph{Introduction to the Modern Theory of Dynamical Systems}. Vol.~54. Encyclopedia of Mathematics and its Applications. Cambridge University Press, Cambridge. 1995.%
\bibitem{KHG} C.~Kawan, S.~Hafstein, P.~Giesl. \emph{A subgradient algorithm for data-rate optimization in the remote state estimation problem}. SIAM Journal on Applied Dynamical Systems 20(4), 2141--2173, 2021.%
\bibitem{LLi} J.~Lawson, Y.~Lim. \emph{Monotonic properties of the least squares mean}. Mathematische Annalen 351, 267--279, 2011.%
\bibitem{LYo} F.~Ledrappier, L.--S.~Young. \emph{The metric entropy of diffeomorphisms: Part I: Characterization of measures satisfying Pesin's entropy formula}. Annals of Mathematics, 509--539, 1985.%
\bibitem{LSe} A.~S.~Lewis, H.~S.~Sendov. \emph{Nonsmooth analysis of singular values. Part I: Theory}. Set-Valued Analysis 13(3), 213--241, 2005.%
\bibitem{Lea} M.~Louzeiro, C.~Kawan, S.~Hafstein, P.~Giesl, J.~Yuan. \emph{A projected subgradient method for the computation of adapted metrics for dynamical systems}. SIAM Journal on Applied Dynamical Systems 21(4), 2610--2641, 2022.%
\bibitem{Luc} D.~T.~Luc. \emph{Theory of Vector Optimization}. Springer, 1989.%
\bibitem{MOA} A.~W.~Marshall, I.~Olkin, B.~C.~Arnold. \emph{Inequalities: Theory of Majorization and its Applications}. 2nd ed. Springer, New York, 2011.%
\bibitem{Mea} Q.~Min, X.~Jian-Sheng, Z.~Shu. \emph{Smooth Ergodic Theory for Endomorphisms}. Lecture Notes in Mathematics, vol.~1978, Springer, Berlin, 2009.%
\bibitem{OGo} H.~Oeri, D.~Goluskin. \emph{Convex computation of maximal Lyapunov exponents}. Nonlinearity, vol.~36, no.~10, 5378, 2023.%
\bibitem{PGo} J.~P.~Parker, D.~Goluskin. \emph{Computation of attractor dimension and maximal sums of Lyapunov exponents using polynomial optimization}. arXiv preprint arXiv:2510.14870, 2025.%
\bibitem{Sch} A.~Schmeding. \emph{An Introduction to Infinite-dimensional Differential Geometry}. Cambridge University Press, vol.~202, 2022.%
\bibitem{Udr} C.~Udriste. \emph{Convex Functions and Optimization Methods on Riemannian Manifolds}. Springer Science \& Business Media 297, 1994.%
\end{thebibliography}
\end{document}